\documentclass[12pt,twoside,a4paper]{amsart}
\usepackage{amsfonts}
\usepackage{amssymb}
\usepackage{amsmath}
\usepackage{amscd}
\usepackage{color}
\usepackage[all]{xy}
\usepackage{mathrsfs}
\usepackage{xspace,yhmath,mathdots,txfonts}
\usepackage{yfonts}

\usepackage{hyperref}

\DeclareMathAlphabet{\mathpzc}{OT1}{pzc}{m}{it}
\numberwithin{equation}{section}
\theoremstyle{plain}

\newtheorem{thm}{Theorem}[section]

\newtheorem{lem}[thm]{Lemma}

\newtheorem{prop}[thm]{Proposition}

\theoremstyle{definition}

\newtheorem{defn}{Definition}[section]

\newtheorem{ntz}{Notation}[section]

\newtheorem{rmk}[thm]{Remark}

\newcommand\opa{\bar{\partial}}
\numberwithin{equation}{section}

\newcommand\C{\mathbb{C}}
\newcommand\R{\mathbb{R}}
\newcommand\Z{\mathbb{Z}}

\newcommand\PP{\mathbb{P}}

\newcommand\Le{\mathscr{L}}

\newcommand\sfW{\textsf{W}}

\newcommand\Tq{\mathpzc{T}}
\newcommand\cO{\mathscr{O}}

\newcommand\Qq{\mathpzc{Q}}

\newcommand\CP{\mathbb{CP}}
\newcommand\Gr{\mathpzc{Gr}}
\newcommand\Fq{\mathpzc{F}}

\newcommand\Pro{\mathbb{P}}

\newcommand\Id{\mathrm{I}}

\newcommand\Ci{\mathcal{C}^\infty}
\newcommand\Cc{\mathcal{C}}

\newcommand\Xs{\textsf{X}}
\newcommand\Bs{\textsf{B}}
\newcommand\sfM{\textsf{M}}

\newcommand\Ta{\mathrm{T}}

\newcommand\Tas{\mathrm{T}^{\ast}{}}
\newcommand\Ha{\mathrm{H}}
\newcommand\CT{\mathbb{C}\mathrm{T}}
\newcommand\Uq{\mathpzc{U}}
\newcommand\Mon{\mathpzc{Mon}}

\newcommand\Jd{\mathrm{J}}

\newcommand\epi{\epsilonup}

\newcommand\pct{\mathrm{p}}
\newcommand\qct{\mathrm{q}}

\newcommand\pq{\mathpzc{p}}
\newcommand\qq{\mathpzc{q}}

\newcommand\cI{\mathcal{I}}

\newcommand\cN{\mathcal{N}}
\newcommand\fq{\mathpzc{f}}
\newcommand\gq{\mathpzc{g}}
\newcommand\vq{\mathpzc{v}}
\newcommand\wq{\mathpzc{w}}
\newcommand\uq{\mathpzc{u}}

\newcommand\sfd{\textsf{d}}
\newcommand\cq{\mathpzc{c}}

\newcommand\Hi{\mathpzc{H}}
\newcommand\Ai{\mathcal{A}}

\newcommand\Ft{\texttt{F}}
\newcommand\Qt{\texttt{Q}}
\newcommand\Kt{\texttt{K}}

\newcommand\Et{\texttt{E}}
\newcommand\Eq{\mathpzc{E}}
\newcommand\Nt{\texttt{N}}
\newcommand\Kq{\mathpzc{K}}
\newcommand\Vs{\textsf{V}}

\newcommand\ttt{\texttt{t}}
\newcommand\id{\mathrm{id}}
\newcommand\Gf{\mathbf{G}}
\newcommand\sq{\mathpzc{s}}
\newcommand\Ot{\Theta}
\newcommand\supp{\mathrm{supp}}

\newcommand\sfX{\textsf{X}}

   \def\DHLhksqrt#1#2{\setbox0=\hbox{$#1\sqrt{#2\,}$}\dimen0=\ht0
     \advance\dimen0-0.2\ht0
     \setbox2=\hbox{\vrule height\ht0 depth -\dimen0}
     {\box0\lower0.4pt\box2}}

\begin{document}

\title{A stability theorem for projective CR manifolds}

\author[J.~Brinkschulte]{Judith Brinkschulte}
\address{J.\ Brinkschulte:
Mathematisches Institut\\ Universit\"at Leipzig\\
Augustusplatz 10/11\\ 04109 Leipzig (Germany)}
\email{brinkschulte@math.uni-leipzig.de}

\author[C.D.~Hill]{C. Denson Hill}
\address{C.D.\ Hill:
Department of Mathematics\\ Stony Brook University
\\ Stony Brook, N.Y. 11794 (USA)}
\email{dhill@math.stonybrook.edu}

\author[M.~Nacinovich]{Mauro Nacinovich}
\address{M.\ Nacinovich:
Dipartimento di Matematica\\ II Universit\`a di Roma
``Tor Ver\-ga\-ta''\\ Via della Ricerca Scientifica\\ 00133 Roma
(Italy)}
\email{nacinovi@mat.uniroma2.it}

\maketitle

\vspace{0.5cm}

\begin{abstract}
We consider smooth deformations of the $CR$ structure of
a smooth  $2$-pseudoconcave compact $CR$ 
submanifold $\sfM$ of a reduced complex analytic variety
$\sfX$ outside the intersection 
$D\,{\cap}\,\sfM$ 
with the support $D$ 
of a Cartier divisor of a positive line bundle $\Ft_{\sfX}.$
We show  that nearby structures still
admit projective $CR$ embeddings. 
Special results are obtained under the additional assumptions that 
$\sfX$ is a projective space or a Fano variety.
\end{abstract}
\tableofcontents
\vspace{0.5cm}

\section{Introduction}

 In this paper we study projective $CR$ manifolds, meaning those that have a $CR$ embedding into some $\C\PP^N$. We
consider a small $CR$ deformation of such an $\sfM$, and investigate under what conditions the
deformed $CR$ manifold still has a projective $CR$ embedding. 
  Our main result is the following $CR$-embedding theorem.

\begin{thm} \label{main}
Let $\sfX$ be an $(n{+}k)$-dimensional  
reduced complex analytic variety and $D$ the support of 
a Cartier divisor of a positive line bundle $\Ft_{\sfX}.$ 
\par 
Let $\sfM$ be a smooth compact submanifold of real dimension $2n{+}k$ of $\sfX_{\text{reg}}$ 
and  $\{\sfM_{\ttt}\}_{\ttt\in\R}$ a  family of $CR$ structures of type $(n,k)$
on $\sfM,$ 
smoothly depending on a real parameter $\ttt.$ We assume that 
\begin{enumerate}
 \item 
 $\sfM_{0}$ is induced by the embedding $\sfM\,{\hookrightarrow}\sfX$;
 \item
 $\sfM_{0}$ is $2$-pseudoconcave;
 \item 
the $CR$ structures of all $\sfM_{\ttt}$ agree to infinite order on $\sfM\,{\cap}D$. 
\end{enumerate}
Then we can find $\epi_{0}{>}0$ such that, for every $|\ttt|\,{<}\,\epi_{0},$
$\sfM_{\ttt}$ admits a generic
$CR$ embedding into a projective complex variety $\sfX_{\ttt}.$ 
\end{thm}

More precise results are obtained under the additional assumptions that 
$\sfX$ is a projective space or a Fano variety (Theorems \ref{projective} and \ref{Fano}).\par
The example of \S\ref{secex} shows that an assumption of $1$-pseudoconcavity
on $\sfM_{0}$ is not sufficient to obtain the statement of Theorem \ref{main}. \\

 Precise definitions and explanation of terms used
will be found in the following sections; let us start by explaining how our main theorem relates to
other results in the field of CR geometry or analysis on CR manifolds.\\

     The notion of a $CR$ structure of type $(n,k)$ on a differentiable manifold  $\sfM$ of real dimension
$2n{+}k$ arises very naturally in two, a priori different, contexts. One is quite geometric in nature, and
the other is more analytical, being connected with fundamental questions about partial differential
equations.\par

     From the geometric perspective, a $CR$ structure of type $(n,k)$ on a smooth real submanifold $\sfM$ of dimension $2n{+}k$ of some complex manifold $\sfX$ of complex dimension $n{+}k$ is induced as the tangential part of the ambient complex structure on $\sfX$. In particular, one has the tangential Cauchy-Riemann equations on
$\sfM$.   In this situation $\sfM$ is said to be generically $CR$
embedded in $\sfX$. It is also possible that
$\sfM$ might be embedded in a complex manifold $X$ of complex dimension greater than $n{+}k$, with
the $CR$ structure on $\sfM$ also being induced as the tangential part of the ambient complex structure
in $\sfX$. In this situation the real codimension of $\sfM$ is greater than the $CR$ dimension $k$, and $\sfM$
is said to be non generically $CR$ embedded in $\sfX$. This is an important distinction only globally,
because: if $\sfM$ can be locally non generically CR embedded, then $\sfM$ can be locally generically
$CR$ embedded (see e.g. \cite{HT}).\par

     From the PDE perspective, one begins with (at least locally) a system of n smooth complex vector
fields defined on $\sfM$, which are linearly independent over $\mathbb{C}$, 
and satisfy a formal integrability condition.
 Again $n$ is the $CR$ dimension, and $k$ is the $CR$ codimension. One then asks if
the given complex vector fields might represent (in a local basis) the tangential $CR$ equations on $\sfM$
given by some (local) $CR$ embedding of $\sfM$.  The criterion for this, is that one must be able to find
(locally) a maximal set of functionally independent characteristic coordinates, which means to find
$n{+}k$ independent complex valued functions that are solutions of the homogeneous system associated
to the $n$ complex vector fields. When the vector fields can be chosen to have real analytic coefficients,
then this is possible (\cite{AH}). Hence locally, in the real analytic case, the geometric perspective and the
analytical PDE perspective are equivalent. The famous example of L. Nirenberg \cite{Ni} showed that the
two perspectives are not locally equivalent if the coefficients of the vector fields are assumed to be
only smooth. Why does this matter? From an analytical PDE viewpoint,
it is too restrictive to allow only real analytic structures, one needs cut off functions, partitions
of unity; etc., in order to employ calculus and modern analysis. So from this perspective it is more
natural to define what is known as an abstract $CR$ manifold $\sfM$ of type $(n,k)$, by reformulating
the local vector field basis description given above  into a global and invariant version, and to do
so within the smooth category. Precise definitions and details are below.\par
 
     The example of Nirenberg \cite{Ni} was a single complex vector field in $\mathbb{R}^3$, 
     so it endows $\mathbb{R}^3$ with the
structure of an abstract $CR$ manifold of type $(1,1)$. To make matters worse, the complex vector field of
Nirenberg's example is a small smooth perturbation, near some point, of the tangential Cauchy
Riemann equation to the unit sphere in $\mathbb{C}^2$. In other words, the generically $CR$ embedded and real
analytic $CR$ structure of type $(1,1)$ on the sphere $S^3$, inherited from $\mathbb{C}^2$, becomes not even locally
$CR$ embeddable near some point, if it is perturbed in a smooth way to obtain a very nearby abstract $CR$
structure of type $(1,1)$ on the same $S^3$. Here the original $CR$ structure on $S^3$ is strictly pseudoconvex,
and remains so under the small perturbation.\par

     Rossi \cite{R} constructed small real analytic deformations of the  $CR$ structure 
     on $S^3$ in $\mathbb{C}^2$ , such that
the resulting perturbed abstract $CR$ structures fail to $CR$ embed globally into $\mathbb{C}^2$, 
even though they do
locally %
$CR$ embed 
into~$\mathbb{C}^2$.  On the other hand, sticking with strictly pseudoconvex $CR$ structures
of hypersurface type, i.e. type $(n,1),$
 Boutet de Monvel proved that when $\sfM$ is compact and $n \,{>}\, 1$,
then $\sfM$ has a $CR$ embedding into $\mathbb{C}^N$ for some perhaps large $N$.  
Catlin and Lempert \cite{CL} showed that there are compact 
strictly pseudoconvex $CR$ manifolds of hypersurface type $(n,1)$, for $n \geq 1$,
that are $CR$ embedded into some $\mathbb{C}^N$, which admit small deformations that are also embeddable,  but
their embeddings cannot be chosen close to the original embedding. Thus these examples exhibit the
phenomenon of 
\textit{instability}
of the $CR$ embedding. On the other hand, for $n\, {=}\, 1$, 
Lempert \cite{L} proved that
if a compact strictly pseudoconvex $\sfM$ is $CR$ embeddable 
into %
$\mathbb{C}^2$, 
hence generically embedded, and if the $CR$ structure gets perturbed 
by a family of $CR$ embeddable $CR$ structures,   then
the embedding is 
\textit{stable}. %
This means that the embedding of the perturbed structure stays close to the
unperturbed one. Additional references to other work on the strictly pseudoconvex hypersurface
type case can be found in the above mentioned papers.\par

     In the higher codimensional case 
($k\,{>}\,1$)     
Polyakov \cite{P} proved a stability theorem for compact
3-pseudoconcave $CR$ manifolds $\sfM$, generically $CR$ embedded in a complex manifold, under
some additional hypotheses, such as the vanishing of a certain cohomology group.  The notion of
flexible/inflexible, which is closely related to the notion of stability/instability, was introduced in
 \cite{BH1}, also for the higher codimension case, where the emphasis was on $CR$ manifolds $\sfM$ that are not
compact. A $CR$ manifold $\sfM$ of type $(n,k)$, generically $CR$ embedded in $\mathbb{C}^{n+k}$, is called flexible
if it admits a compactly supported $CR$ deformation whose %
 deformed structure is no
longer $CR$ embeddable in $\mathbb{C}^{n+k}$; in other words, the deformation causes the $CR$ structure to
{\textquotedblleft{flex out}\textquotedblright}
of the space 
where the original manifold lies. %
On the other hand, $\sfM$ is inflexible
if any compactly supported deformation stays in the 
class of manifolds globally $CR$ embeddable into $\C^{n+k}.$ %
In \cite{BH1}, \cite{BH2} it was proved that if $\sfM$ is a  $CR$ submanifold of type $(n,k)$ in $\mathbb{C}^{n+k}$, 
and it is
only 2-pseudoconcave, then M is inflexible. 
 This relates to the present paper as follows: If a $CR$ structure of a 
 2-pseudoconcave compact $CR$ manifold of a reduced complex analytic variety  
 gets perturbed by a smooth family of $CR$ structures that 
 {\textquotedblleft{glues to}\textquotedblright} the unperturbed structure 
 along a Cartier divisor of a positive line 
 bundle, then the perturbed $CR$ structures {\it do not} 
{\textquotedblleft{flex out}\textquotedblright} of the 
class %
of globally $CR$ embeddable $CR$ manifolds (Theorem~\ref{main}). Moreover, 
the embedding is stable if the unperturbed $CR$ manifold sits generically in a projective space (Theorem \ref{projective}).
\par
     
We emphasize that      
in this paper we discuss compact $CR$ manifolds and we allow higher
codimension and that we are interested in global results, not local ones.\par

\section{Preliminaries and notation}\label{secnot}
\subsection{$CR$ manifolds of type $(n,k)$} \label{subs2.1}
 Let $\sfM$ be a smooth manifold of real dimension $m$
and $n,k$ nonnegative integers with $2n{+}k{=}m.$ 
A $CR$ structure of type $(n,k)$ on $\sfM$ is a formally integrable distribution
$\Ta^{0,1}_{\sfM}$ 
of $n$-dimensional complex 
subspaces of the complexified tangent bundle $\CT_{\sfM}$ of $\sfM,$
transversal to~$\Ta_{\sfM}.$   
By this transversality condition,  the real parts of  vectors in $\Ta^{0,1}_{\sfM}$
are the elements of the \emph{vector $CR$ distribution} $\Ha_{\sfM}$ and 
for each $\pct\,{\in}\,\sfM$ and
$\vq\,{\in}\,\Ha_{\sfM,\pct}$ 
(subscript {\textquotedblleft{$\pct$}\textquotedblright} means 
 {\textquotedblleft{the fiber at $\pct$\textquotedblright})
there is a unique $\vq'\,{\in}\,\Ha_{\sfM,\pct}$ such that
$\vq{+}i\vq'\,{\in}\,\Ta^{0,1}_{\sfM{},{\pct}}.$ 
The correspondence $\vq{\to}\vq'$ defines a  bundle self
map $\Jd_{\sfM}\,{:}\,\Ha_{\sfM}\,{\to}\,\Ha_{\sfM}$ which yields a complex structure on the fibers of
$\Ha_{\sfM},$ as $\Jd_{\sfM}^{2}\,{=}\,{-}\Id_{\Ha_{\sfM}}.$ The formal integrability condition 
means that 
\begin{equation}\label{e2.1}
 [\Gamma^{\infty}(\sfM,\Ta^{0,1}_{\sfM}),\Gamma^{\infty}(\sfM,\Ta^{0,1}_{\sfM})]
 \subseteq\Gamma^{\infty}(\sfM,\Ta^{0,1}_{\sfM}),
\end{equation}
which is equivalent to the vanishing 
of the  
 Nijenhuis tensor on $\Ha_{\sfM}$, i.e. 
\begin{equation*}
 \Jd_{\sfM}([X,Y]-[\Jd_{\sfM}X,\Jd_{\sfM}Y])=[\Jd_{\sfM}X,Y]+[X,\Jd_{\sfM}Y],\;\;\forall \,X,Y\in
 \Gamma^{\infty}(\sfM,\Ha_{\sfM}).
\end{equation*}
 We use the standard notation 
\begin{align*}
 \Ta^{1,0}_{\sfM}&=\{X-i\Jd_{\sfM}X\mid X\in{\Ha}_{\sfM}\}\subset\C\Ta_{\sfM},\\
  \Ta^{0,1}_{\sfM}&=\{X+i\Jd_{\sfM}X\mid X\in{\Ha}_{\sfM}\}\subset\C\Ta_{\sfM},\\
  \Ta^{\ast}{}^{1,0}_{\sfM}&=\{\phiup\in{\CT}^{\ast}_{\sfM}\mid \phiup(Z)=0,\; 
  \forall Z\in\Gamma^{\infty}(\sfM,\Ta^{0,1}_{\sfM})\},\\
    \Ta^{\ast}{}^{0,1}_{\sfM}&=\{\phiup\in{\CT}^{\ast}_{\sfM}\mid \phiup(Z)=0,\; 
    \forall Z\in\Gamma^{\infty}(\sfM,\Ta^{1,0}_{\sfM})\},\\
    \Ha^{0}_{\sfM}&=\{\xiup\in{\Ta}^{\ast}_{\sfM}\mid \xiup(X)=0,\;\forall X\in\Gamma^{\infty}(\sfM,{\Ha}_{\sfM})\}.
    \intertext{We have}
 \Ta^{1,0}_{\sfM}\cap {\Ta}^{0,1}_{\sfM} & {=}\{0\},\;\;
  \Ta^{1,0}_{\sfM}{+} {\Ta}^{0,1}_{\sfM}  {=}\C\Ha_{\sfM},\;\;  
    \Ta^{\ast}{}^{1,0}_{\sfM}{\cap}  \Ta^{\ast}{}^{0,1}_{\sfM}{=}\C{{\Ha}}_{\sfM}^{0},\;\;
     \Ta^{\ast}{}^{1,0}_{\sfM}{+}  \Ta^{\ast}{}^{0,1}_{\sfM}{=} \C\Ta^{\ast}_{\sfM}.  
\end{align*}
\par\smallskip

Deformations of $CR$ structures were considered especially for the 
hy\-per\-sur\-face-type case $k=1$
\cite{Ak, DBM}. Since, by a theorem of Darboux (see \cite{D}),
$1$-codimensional contact structures 
are locally equivalent, deformations of non degenerate $CR$ hypersurfaces 
can be  described  by   smooth curves
$\{\ttt{\to}\Jd_{\ttt}\}$ of formally integrable complex structures on
a fixed contact distribution ${\Ha}_{\sfM}$ of 
a smooth real manifold $\sfM.$ 
Differently,  generalized contact distributions of higher codimension may
be not locally equivalent. Therefore, 
to describe general deformations of abstract
$CR$ manifolds of type $(n,k)$ with arbitrary $k$ 
it will be convenient to focus our attention on
their 
\textit{canonical bundles}, which carry comprehensive
information on both their contact and  partial complex structures,
as explained below. 
\par 

\smallskip

Denote by  
$\Ai_{\sfM}$ the sheaf of
Grassmann algebras of germs of 
complex valued smooth exterior differential forms on $\sfM.$ 
The subsheaf
$\Ai^{\qq}_{\sfM}$ of
homogeneous elements of degree $\qq$ consists of 
the germs of smooth 
sections of the  $\C$-vector bundle $\Lambda^{\qq}(\C\Ta^{\ast}_{\sfM})$ 
of complex $\qq$-covectors on~$M.$ 
 \par 
We can identify 
the Grassmannian  $\Gr_{\!\!n}^{\C}(\sfM)$ of $n$-dimensional complex subspaces
of $\C\Ta_{\sfM}$ with the subbundle of the  
projective bundle $\PP(\Lambda^{n+k}(\C\Ta^{\ast}_{\sfM}))$
whose representatives belong to the cone   
$\Mon(\Lambda^{n+k}(\CT^{\ast}M))$ of degree $n{+}k$ 
monomials\footnote{We recall that, if $\Vs$ is a vector space, the \emph{monomials of degree $\qq$} 
of its exterior algebra
$\Lambda^{*}\Vs$ are the exterior products $\vq_{1}{\wedge}\cdots{\wedge}\vq_{\qq}$ of
vectors $\vq_{1},\hdots,\vq_{\qq}$ of $\Vs.$}
of the Grassmann algebra
of $\CT^{\ast}M$ (see e.g. \cite{C}).
Then  a distribution of complex $n$-planes of $\CT_{\sfM}$ is the datum of 
a smooth complex line bundle $\Kt_{\sfM}$ on $\sfM,$ 
with $\Kt_{\sfM}\,{\subset}\,\Mon(\Lambda^{n+k}(\CT^{\ast}M)).$
\begin{defn} Let $(M,\Ta^{0,1}_{\sfM})$ be an abstract $CR$ manifold of type $(n,k).$ 
 The complex line bundle $\Kt_{\sfM}$ corresponding to the distribution $\Ta^{0,1}_{\sfM}$
 of its \textit{anti-holomorphic tangent vectors} is called its \emph{canonical bundle}. 
\end{defn}
\begin{ntz}
 Let us denote by $\Kq_{\sfM}$ the subsheaf of ideals in $\Ai_{\sfM}$ generated by the
 sheaf  ${\underline{\Kt}}{}_{\sfM}$ of germs of smooth sections of $\Kt_{\sfM}.$ 
\end{ntz}
\begin{prop} \label{defn2.1}
A $CR$ structure of type $(n,k)$ on $\sfM$ is completely determined by the
datum of 
a complex line bundle $\Kt_{\sfM}\,{\subset}\,\Mon(\Lambda^{n+k}(\CT^{\ast}_{\sfM}))$
such that
 the ideal sheaf 
 $\Kq_{\sfM}$  of $\Ai_{\sfM}$ generated by the germs its 
 smooth
 sections  satisfies 
the following conditions: 
\begin{itemize}
 \item[(1)] $\Kq_{\sfM}$ is \textit{closed}, i.e. 
 $d\Kq_{\sfM}\subseteq\Kq_{\sfM}$ \quad {(\textsc{formal integrability});}
 \item[(2)] $\{X\in\Gamma^{\infty}(M,{\Ta}_{\sfM})\,{\mid}\, X\,{\rfloor}\,\Kq_{\sfM}
 \,{\subseteq}\,\Kq_{\sfM}\}\,{=}\,\{0\}$\quad
 (\textsc{transversality to $\Ta_{\sfM}$}).
\end{itemize}
\par 
Indeed, when $(1)$ and $(2)$ are valid, 
$\Kt_{\sfM}$ is the canonical bundle of the $CR$ manifold $(\sfM,\Ta_{\sfM}^{0,1})$ 
with 
\begin{equation}
 \Gamma^{\infty}(\sfM,\Ta_{\sfM}^{0,1})=\{Z\in\Gamma^{\infty}(\sfM,\C\Ta_{\sfM})\mid
 Z\rfloor\Kq_{\sfM}\,\subseteq\,\Kq_{\sfM}\} 
\end{equation}
\end{prop} 
\begin{proof}
 In fact, $(1)$  is equivalent to the formal integrability \eqref{e2.1}
 and $(2)$ to the transversality of $\Ta^{0,1}_{\sfM}$ to the real tangent distribution
 $\Ta_{\sfM}.$ 
\end{proof}
By the Newlander-Nirenberg theorem (\cite{NN}) abstract $CR$ manifolds of type $(n,0)$ 
are $n$-dimensional complex manifolds. 

\subsection{Natural projection onto the canonical bundle} On a $CR$ manifold $\sfM$ of
type $(n,k)$ we have an exact sequence of smooth bundle maps \smallskip
\begin{equation} \label{e2.3}
\begin{CD}
 0 @>>> \C\Ha^{0}_{\sfM} @>>> \Ta^{\ast}{}^{1,0}_{\!\!\sfM}\oplus\Ta^{\ast}{}^{0,1}_{\!\!\sfM}
 @>>> \C\Ta^{\ast}_{\sfM}@>>>0,
\end{CD}
\end{equation}
where 
\begin{gather*}
\C\Ha^{0}_{\sfM}\,{\ni}\,\xiup\,\longrightarrow\,(\xiup,\xiup)
\,{\in}\,\Ta^{\ast}{}^{1,0}_{\!\!\sfM}\oplus\Ta^{\ast}{}^{0,1}_{\!\!\sfM},\\
\Ta^{\ast}{}^{1,0}_{\!\!\sfM}\oplus\Ta^{\ast}{}^{1,0}_{\!\!\sfM}\,{\ni}\,(\xiup_{1},\xiup_{2})
\longrightarrow \xiup_{1}{-}\xiup_{2}\,{\in}\, \C\Ta^{\ast}_{\sfM}.
\end{gather*}
A \emph{$CR$ gauge} is a bundle map 
\begin{equation}\label{e2.4}
 \lambdaup:\C\Ta^{\ast}_{\sfM} \longrightarrow\Ta^{\ast}{}^{1,0}_{\!\!\sfM},\;\;\text{such that}\;\;
 \xiup-\tfrac{1}{2}\lambdaup(\xiup)\in\Ta^{\ast}{}^{0,1}_{\!\!\sfM},\;\forall\xiup\in\C\Ta^{\ast}_{\sfM}.
\end{equation}
Note that, in particular, 
\begin{equation*}
 \lambdaup(\Ta^{\ast}{}^{0,1}_{\!\!\sfM})\subseteq\C\Ha^{0}_{\sfM}\;\;\;\text{and}\;\;\;
 \xiup-\tfrac{1}{2}\lambdaup(\xiup)\in\C\Ha^{0}_{\sfM},\;\forall \xiup\in\Ta^{\ast}{}^{1,0}_{\!\!\sfM}.
\end{equation*}
Therefore a $CR$ gauge yields a short exact sequence inverting  \eqref{e2.3}:
\begin{equation*}\begin{CD}
 0 @>>> \C\Ta^{\ast}_{\sfM}@>{(\lambdaup/2)\oplus(\id-\lambdaup/2)}>>
 \Ta^{\ast}{}^{1,0}_{\!\!\sfM}\oplus \Ta^{\ast}{}^{0,1}_{\!\!\sfM} 
 @>{(\lambdaup/{2}-\id,\lambdaup/2)}>> \C\Ha^{0}_{\sfM}@>>>0,
 \end{CD}
\end{equation*}
where $\id$ denotes the identity on $\C\Ta^{\ast}_{\sfM}.$ \par 
There are several possible choices of $\lambdaup.$ 
A $CR$ gauge $\lambdaup$ is said to be \emph{of the real type} if 
\begin{equation} \lambdaup(\xiup)+\overline{\lambdaup(\,\overline{\xiup}\,)}=2{\cdot}\xiup,\;\;
\forall\xiup\,\in\C\Ta^{\ast}{}_{\sfM}
\end{equation}
and \emph{balanced} if 
\begin{equation}
 \lambdaup(\xiup)=\xiup,\;\;\;\forall \xiup\in\Ha^{0}_{\sfM}.
\end{equation} 
Existence of balanced $CR$ gauges of the real type was 
proved in \mbox{\cite[\S{2}A]{MN}}. The following is then easy to deduce from the above properties.
\begin{prop} The bundle map $\lambdaup$ of a balanced $CR$ gauge 
\begin{itemize}
 \item is a projection of $\C\Ta^{\ast}_{\sfM}$ onto $\Ta^{\ast}{}^{1,0}_{\sfM}$;
 \item its restriction to $\Ta^{\ast}_{\sfM}$ is an isomorphism with a real complement
 of $i{\cdot}\Ha^{0}_{\sfM}$ in $\Ta^{\ast}{}^{1,0}_{\!\!\sfM}.$  
\end{itemize}
\end{prop}
Let us define a linear map
\begin{equation}
 \lambdaup^{\pq}:\Lambda^{\pq}(\C\Ta^{\ast}_{\sfM})\to\Lambda^{\pq}(\Ta^{\ast}{}^{1,0}_{\!\!\sfM})
\end{equation}
in such a way that on the monomials we have 
\begin{equation*}
 \lambdaup^{p}(\xiup_{1}\wedge\cdots\wedge\xiup_{\pq})=\lambdaup(\xiup_{1})\wedge
 \cdots\wedge\lambdaup(\xiup_{\pq}).
\end{equation*}
\begin{prop}  The map
$\lambdaup^{n+k}$ defines a projection of $\Lambda^{n+k}(\C\Ta^{\ast}_{\sfM})$ onto $\Kt_{\sfM},$
which is independent of the particular choice of a balanced $CR$ gauge
$\lambdaup$ of the real type.
\end{prop} 
\begin{proof} 
 Fix a point $\pct$ of $\sfM$ and a basis 
\begin{equation*}\tag{$\ddag$}
 \xiup_{1},\hdots,\xiup_{n},\tauup_{1},\hdots,\tauup_{k},\zetaup_{1},\hdots,\zetaup_{n}
\end{equation*}
of $\C\Ta^{\ast}_{\sfM,\pct}$ with $\xiup_{i}\,{\in}\,\Ta^{\ast}{}^{1,0}_{\!\!\sfM,\pct},$
for $1{\leq}i{\leq}n,$ $\tauup_{j}\,{\in}\,\Ha^{0}_{\sfM}$ for $1{\leq}j{\leq}k$ and   
$\zetaup_{i}\,{\in}\,\Ta^{\ast}{}^{0,1}_{\!\!\sfM,\pct}$ for $1{\leq}i{\leq}n.$ 
\par 
If $\lambdaup$ is a balanced $CR$ gauge of the real type, then 
$\lambdaup(\xiup_{i})\,{=}\,\xiup_{i}$ for $1{\leq}i{\leq}n,$
$\lambdaup(\tauup_{j})\,{=}\,\tauup_{j}$ for $1{\leq}j{\leq}k,$
$\lambdaup(\zetaup_{i})\,{\in}\,\C\Ha^{0}_{\sfM}$ for $1{\leq}i{\leq}n.$
Hence the image by $\lambdaup^{n+k}$ of all monomials of degree $n{+}k$ 
in the elements of the basis
$(\ddag)$ which contain a factor $\zetaup_{i}$ contain at least $k{+}1$ factors in $\C\Ha^{0}_{\sfM}$
and is therefore $0,$ while 
\begin{equation*}
 \lambdaup^{n+k}(\xiup_{1}\wedge\cdots\wedge\xiup_{n}\wedge\tauup_{1}\wedge\cdots\wedge\tauup_{k})
 =\xiup_{1}\wedge\cdots\wedge\xiup_{n}\wedge\tauup_{1}\wedge\cdots\wedge\tauup_{k}.
\end{equation*}
This completes the proof.
\end{proof}
\begin{defn}
 We call $\lambdaup^{n+k}$ the \emph{natural projection} onto the canonical bundle.
\end{defn}

\subsection{$CR$ maps}
Let  $(\sfM_{i},\Kt_{\sfM_{i}})$ be abstract $CR$  manifolds of type $(n_{i},k_{i})$ ($i{=}1,2$),
 $f\,{:}\,\sfM_{1}\,{\to}\,\sfM_{2}$ a smooth map and $f_{*}\,{:}\,
 \C\Ta_{\sfM_{1}}\,{\to}\,\C\Ta_{\sfM_{2}}$ the bundle map
obtained by complexifying its differential.
\begin{defn}
The map $f$ is said to be:
\begin{itemize}
 \item $CR$ if  $f_{*}(\Ta^{0,1}_{\sfM_{1}})\,\subseteq\,\Ta^{0,1}_{\sfM_{2}}$;
\item a $CR$ \emph{immersion} (\emph{embedding}) if it is a smooth immersion (embedding)
and moreover $\Ta^{0,1}_{\sfM_{1}}\,{=}\,\{Z\in\C\Ta_{\sfM_{1}}\,{\mid}\, f_{*}(Z)\,{\in}\,
\Ta^{0,1}_{\sfM_{2}}\}.$ 
\item a \emph{$CR$ submersion} if it is a smooth submersion and $f_{*}(\Ta^{0,1}\sfM_{1})
{=}\Ta^{0,1}\sfM_{2}.$ 
\end{itemize}\par 
A $CR$ immersion (embedding) is called \emph{generic} when $n_{1}{+}k_{1}{=}n_{2}{+}k_{2}.$   
\end{defn}

On an abstract
$CR$ manifold $\sfM$ of type $(n,k),$ we consider 
the Grassmann subalgebra $\Omega^{*}_{\sfM}$
of $\Lambda^{*}(\C\Ta^{\ast}_{\sfM})$ generated by
$\Ta^{\ast}{}^{1,0}_{\sfM}$:
\begin{equation*}
 \Omega^{\ast}_{\sfM}={\sum}_{\pq=0}^{n+k}\Omega^{\pq}_{\sfM},\quad \text{with\;\;
 $\Omega^{\pq}_{\sfM}=\Omega^{\ast}_{\sfM}\cap\Lambda^{\pq}(\C\Ta^{\ast}_{\sfM})$}.
\end{equation*}
With this notation, $\Kt_{\sfM}\,{=}\,\Omega_{\sfM}^{n+k}.$ 
$CR$ maps and immersion can be described by using the bundles $\Omega^{*}_{\sfM}.$ 

\begin{prop} Let $\sfM_{i}$ be abstract $CR$ manifolds of type $(n_{i},k_{i}),$ for $i{=}1,2$ 
and $f\,{:}\,\sfM_{1}\,{\to}\,\sfM_{2}$ a smooth map. Then $f$ is 
\begin{itemize}
 \item $CR$ iff $f^{*}(\Omega^{1}_{\sfM_{2}})
\,{\subseteq}\,\Omega^{1}_{\sfM_{1}}$;
\item a $CR$ immersion iff $f$ is a smooth immersion and
$f^{*}(\Omega^{n_{1}+k_{1}}_{M_{2}})\,{=}\,\Kt_{M_{1}}.$ 
\end{itemize}  
A generic $CR$-immersion $f\,{:}\,M_{1}\,{\to}\,M_{2}$
is  characterized by 
\begin{equation}\vspace{-20pt}
f^{*}(\Kt_{M_{2}})=\Kt_{M_{1}}.
\end{equation} \qed
\end{prop}
\par
\vspace{2pt}

\subsection{$\qq$-pseudoconcavity}
Let $\qq$ be a nonegative integer. 
Following \cite{HN1}, we say that 
a $CR$ manifold 
$M,$ of type $(n,k),$ 
is \emph{$\qq$-pseu\-do\-con\-cave}
if, for every $\pct\,{\in}\, M$ 
and every characteristic covector $\xiup\,{\in}\, 
\Ha^{0}_{\pct} M
{\backslash}\{0\}$, 
the scalar
Levi form $\Le_{\pct}(\xiup, \cdot)$ 
has at least $q$ negative and $q$ positive eigenvalues.\par
We recall that 
\begin{equation*}
 \Le_{\pct}(\xiup,\vq)=\xiup([\Jd_{\sfM}X,X])=\sfd\,{\tilde{\xiup}}(\vq,\Jd_{\sfM}(\vq)),\;\; \text{for}\, \vq\in\Ha_{\pct}M,
\end{equation*}
with $X\,{\in}\,\Gamma^{\infty}(M,\Ha_{\sfM}),$ $\tilde{\xiup}\in\Gamma^{\infty}(M,\Ha^{0}M)$
and $X_{\pct}\,{=}\,\vq,$ $\tilde{\xiup}_{\pct}=\xiup.$ \par
The $CR$ dimension $n$ of a $\qq$-pseudoconcave $CR$ manifold is  at least $2\qq$ 
(more precise lower bounds on $n,$ depending also on the $CR$ codimension $k,$
can be obtained e.g. from \cite{AY}).
\par 
Various equivalent definitions of the Levi form and more of its aspects are explained in~\cite{BHLN}.\\

\subsection{Deformation of $CR$ structures}
With the preparation of \S\ref{subs2.1}, we can define 
$CR$ deformations
of a $CR$ manifold $\sfM$ of general
type $(n,k)$ in terms of canonical bundles,
allowing 
in this way 
also a deformation of  the underlying contact structure $\Ha_{\sfM}.$

\begin{defn}\label{defn2.3}
 A smooth one-parameter deformation of a $CR$ structure $\Kt_{\sfM}$ on $\sfM$ 
 is a smooth map $\sfM{\times}(-\epi_{0},\epi_{0})\ni(\pct,\ttt)\,{\to}\,[\Kt_{\sfM_{\ttt},\pct}]\,{\in}
 \Pro(\Lambda^{n+k}({\CT}^{\ast}))$ such that, for each $\ttt\,{\in}\,(-\epi_{0},\epi_{0}),$ 
$\Kt_{\sfM_{\ttt}}$ is the canonical complex line bundle 
of a $CR$ structure on $\sfM$ 
 and $\Kt_{\sfM_{0}}\,{=}\,\Kt_{\sfM}.$ 
\end{defn}
Line bundles over a contractible base are trivial (see e.g. \cite[Cor.3.4.8]{Hu}). 
Hence, for
a compact $\sfM,$ we can find a finite open covering  $\Uq\,{=}\,\{U_{j}\}$ 
to describe the deformation by the data of smooth maps 
\begin{equation}
 \sfM\times (-\epi_{0},\epi_{0})\ni (\pct,\ttt) \to \omegaup_{j}(\pct,\ttt)\in\Lambda^{n+k}\C\Tas{\sfM}
\end{equation}
such that 
  $\omegaup_{j,\ttt}(\,\cdot\,,\ttt)
 \in\Gamma^{\infty}(U_{j},\Kt_{M_{t}})$
  for each $\ttt\,{\in}\,(-\epi_{0},\epi_{0}).$

\section{$\bar{\partial}_{\sfM}$-cohomology}\label{secprel}

Let $\sfM$ be an abstract $CR$ manifold of type $(n,k).$ 
\par 
The ideal sheaf $\cI_{\!\sfM}$ of $\Ai_{\sfM}$ generated by the germs of smooth
sections of $\Ta^{\ast}{}^{1,0}_{\!\!\sfM}$ is, according to \cite{C}, 
the \emph{characteristic system} of $\Ta^{0,1}_{\sfM}.$ Hence, formal
integrability is equivalent to 
\begin{equation*}
 \sfd(\cI_{\!\sfM})\subseteq\cI_{\!\sfM}.
\end{equation*}
For nonnegative integers $\pq,$ 
its exterior powers $\cI_{\!{\sfM}}^{p}$ are also  $\sfd$-closed ideals
and we  can consider the quotients 
of the de Rham 
complex:
\begin{equation*}\tag{$*$}
 \opa_{\sfM} : \cI_{\!{\sfM}}^{\pq}/ \cI_{\!{\sfM}}^{\pq+1} \longrightarrow  
\cI_{\!{\sfM}}^{\pq}/ \cI_{\!{\sfM}}^{\pq+1}
\end{equation*}
(we put $\cI_{\!\sfM}^{0}{=}\Ai_{\sfM}$). Since  
$\cI_{\!{\sfM}}^{p}{=}0$ when $\pq{>}n{+}k,$ the map
 $(*)$ is nontrivial if and only if $0{\leq}\pq{\leq}n{+}k.$  

\par 
The canonical $\Z$-grading of $\Ai_{\sfM}$ induces a $\Z$-grading 
$\cI_{\!{\sfM}}^{\pq}\,{=}\,{\bigoplus}_{\qq=0}^{n}\cI_{\!{\sfM}}^{\pq,\qq}$ of $\cI_{\!{\sfM}}^{\pq},$
with 
$\cI_{\!{\sfM}}^{p,q}=\cI_{\!{\sfM}}^{p} \,{\cap}\,\Ai^{p+q}_{\sfM}.$
Let us set $\Qq^{p,\ast}_{\,\sfM}\,{=}\, \cI_{\!{\sfM}}^{\pq}/ \cI_{\!{\sfM}}^{\pq+1}.$ 
By passing to the quotients, the $\Z$-gradings of the exterior powers of $\cI_{\!\sfM}$ 
induce the $\Z$-gradings 
\begin{equation*}
 \Qq^{p,*}_{\,\sfM}\,{=}\,
{\bigoplus}_{q=0}^{n}\Qq^{p,q}_{\,\sfM}.
\end{equation*}
For each pair of integers $(\pq,\qq)$ with $0{\leq}\pq{\leq}n{+}k,$ $0{\leq}\qq{\leq}n$
the summand $\Qq_{\,\sfM}^{\pq,\qq}$ 
is the sheaf of germs of smooth sections of a complex vector bundle
$\Qt_{\sfM}^{\pq,\qq}$ on~$\sfM.$ 
Then for each $0{\leq}\pq{\leq}n{+}k,$ the map $(*)$ is a \textit{differential complex}
\begin{equation*} 
\begin{CD}
 0 \to \Qq^{\pq,0}_{\,\sfM}@>{\bar{\partial}_{\,\sfM}}>>
 \Qq^{\pq,1}_{\,\sfM}@>>> \cdots @>{\bar{\partial}_{\,\sfM}}>>
 \Qq^{\pq,n-1}_{\,\sfM}@>{\bar{\partial}_{\,\sfM}}>>
 \Qq^{\pq,n}_{\,\sfM}\to 0,
\end{CD}
\end{equation*}
which is called 
the \emph{tangential Cauchy-Riemann complex  in degree $\pq$}
(for more details see e.g. \cite{HN1}).\par
 \smallskip
 Since $\cI_{\!\sfM}^{n+k+1}{=}\underline{0},$ the bundles
 $\Qt^{n+k,\qq}_{\,\sfM}$ 
 are subbundles 
 of $\Lambda^{n+k+\qq}(\C\Ta^{\ast}_{\sfM})$ and in particular  
$\Qt^{n+k,0}_{\sfM}$ is the same as
the
\textit{canonical bundle} $\Kt_\sfM$. As we will explain below, 
it is a \emph{$CR$ line bundle} and 
the tangential 
Cauchy-Riemann complex in degree $n{+}k$ is also
its $\bar{\partial}_{\sfM}$-complex as a $CR$ line bundle. 
This is a special example of a notion that was described e.g.
in 
\cite[\S{7}]{HN} and that we will quickly recall in the next subsection. 
\par \medskip
The tangential Cauchy-Riemann complexes can be described in a way that is suitable 
to deal with smooth deformations of the $CR$ structure. \par
Let us set $\Ta^{\ast}{}^{\pq,0}_{{\!\!\sfM}}=\Lambda^{\pq}(\Ta^{\ast}{}^{1,0}_{\!\!\sfM})$
and denote by $\Tq^{\ast}{}^{\pq,0}_{\!\!\sfM}$ 
the sheaf of germs of smooth sections of $\Ta^{\ast}{}^{\pq,0}_{{\!\!\sfM}}.$
\subsection{Differential presentation of the tangential $CR$ complex}
\label{sub3.1a}
Let $\Et_{\sfM}$ be  a complement of $\Ta^{\ast}{}^{1,0}_{\!\!\sfM}$ 
in $\C\Ta^{\ast}_{\sfM}.$ 
We can e.g. fix a smooth Hermitian product on the fibers of $\C\Ta^{\ast}_{\sfM}$
and take $\Et_{\sfM}=(\Ta^{\ast}{}^{1,0}_{\!\!\sfM})^{\perp}.$ The fiber bundle $\Et_{\sfM}$ has
rank $n$ and we obtain 
a bigradation  
of the Grassmann  algebra of the complexified cotangent bundle of $\sfM$ by setting
\begin{equation*}
 \Lambda^{*}(\C\Ta^{\ast}_{\sfM})={\sum}_{
\begin{smallmatrix}
 0{\leq}\pq{\leq}n{+}k\\
 0{\leq}\qq{\leq}n
\end{smallmatrix}
 } \Et^{\pq,\qq}_{\sfM},\;\;\;\text{with}\;\; \Et^{\pq,\qq}_{\sfM}=\Lambda^{\qq}(\Et_{\sfM})
 \otimes_{\sfM}\Ta^{\ast}{}^{\pq,0}_{\!\!\sfM}.
\end{equation*}
Denote by $\Eq^{\pq,\qq}_{\sfM}$ the sheaf of germs of smooth sections of $\Et^{\pq,\qq}_{\sfM}.$ 
\begin{lem} For every integer $\pq$ with $0{\leq}\pq{\leq}n{+}k$ we have: 
\begin{enumerate}
 \item $\Eq^{\pq,\qq}_{\sfM}\,{\subseteq}\,\cI_{\sfM}^{\pq}$\, .
 \item The restriction of the projection onto the quotient defines  isomorphisms 
\begin{equation}\label{e3.1}
 \Et^{\pq,\qq}_{\sfM}\to\Qt^{\pq,\qq}_{\sfM},\;\;\ \Eq^{\pq,\qq}_{\sfM}\to\Qq^{\pq,\qq}_{\sfM}
\end{equation}
\end{enumerate}
of fiber bundles and of their sheaves of sections.  \qed
\end{lem}
By using the isomorphisms of \eqref{e3.1} we can rewrite 
the tangential Cauchy-Riemann complexes
in the form 
\begin{equation*} 
\begin{CD}
 0 @>>>\Eq^{\pq,0}_{\sfM}@>{\sfd''_{\sfM}}>> \Eq^{\pq,1}_{\sfM}@>{\sfd''_{\sfM}}>> \cdots\to
 \Eq^{\pq,n-1}_{\sfM}
 @>{\sfd''_{\sfM}}>> \Eq^{\pq,n}_{\sfM}@>>>0.
\end{CD}
\end{equation*}
\par 
We get explicit representations of $\sfd''_{\sfM},$ and hence of $\Qq^{\pq,\qq}_{\,\sfM}$ and
$\bar{\partial}_{\sfM},$
by choosing
suitable
coframes of $\C\Ta^{\ast}_{\sfM}.$ Take indeed on 
an open subset $U$ of $\sfM$ a coframe 
$\omegaup_{1},\hdots,\omegaup_{n},\etaup_{1},\hdots,
\etaup_{n+k}$ of $\Ta^{\ast}_{\sfM}$
 \par\centerline{with  
$\omegaup_{1},\hdots,\omegaup_{n}$ in $\Eq_{\sfM}$ 
and $\etaup_{1},\hdots,\etaup_{n+k}\,{\in}\,\Gamma^{\infty}(U',\Ta^{\ast}{}^{1,0}_{{\!\!\sfM}}).$}
\par\noindent
Then there are uniquely determined smooth complex vector fields
$\bar{L}_{1},\hdots,\bar{L}_{n},$ $P_{1},\hdots,P_{n+k}$
 such that 
 \begin{equation*}
 \sfd{\,} u={\sum}_{j=1}(\bar{L}_{j}u)\omegaup_{j}+{\sum}_{i=1}^{n+k}(P_{i}u)\etaup_{i}
\end{equation*}
The vector fields $\bar{L}_{h},$ for $1{\leq}h{\leq}n,$ are characterized by 
\begin{equation*}
 \omegaup_{j}(\bar{L}_{h})=\deltaup_{j,h}\,\text{(Kronecker delta)},\; 1{\leq}j{\leq}n,\;\;\;
 \etaup_{i}(\bar{L}_{h})=0,\;\; 1{\leq}i{\leq}n{+}k,
\end{equation*}
and hence  are sections of $\Ta^{0,1}_{\sfM}.$ By our definition, 
\begin{equation*}
 \sfd''_{\sfM}u={\sum}_{j=1}^{n}(\bar{L}_{j}u)\,\omegaup_{j},\;\;\forall u\in\Ci(U).
\end{equation*}
For each $j$ with $1{\leq}j{\leq}n$ we have 
\begin{equation*}
 \sfd\omegaup_{j}={\sum}_{1{\leq}j_{1}{\leq}j_{2}\leq{n}}\cq_{j}^{j_{1},j_{2}}\omegaup_{j_{1}}\wedge
 \omegaup_{j_{2}}+\zetaup_{j},
\end{equation*}
 with $\cq_{j}^{j_{1},j_{2}}\in\Ci(U),$ $\zetaup_{j}\in\cI_{\sfM}(U).$ This gives in particular 
\begin{equation*}
 \sfd''_{\sfM}\omegaup_{j}={\sum}_{1{\leq}j_{1}{\leq}j_{2}\leq{n}}\cq_{j}^{j_{1},j_{2}}\omegaup_{j_{1}}\wedge
 \omegaup_{j_{2}}
\end{equation*}
and hence, computing by recurrence on the number of factors, we get 
\begin{gather*}
\sfd''_{\sfM}(\omegaup_{j_{0}}{\wedge}\omegaup_{j_{1}}{\wedge}\cdots{\wedge}\omegaup_{j_{\qq}})
=\sfd''_{\sfM}(\omegaup_{j_{0}}){\wedge}\omegaup_{j_{1}}{\wedge}\cdots{\wedge}\omegaup_{j_{\qq}}+
\omegaup_{j_{0}}{\wedge}\sfd''_{\sfM}(\omegaup_{j_{1}}{\wedge}\cdots{\wedge}\omegaup_{j_{\qq}}),
\\
\sfd''_{\sfM}(\omegaup_{j_{1}}{\wedge}\cdots{\wedge}\omegaup_{j_{\qq}}\wedge
\etaup_{i_{1}}{\wedge}\cdots{\wedge}\etaup_{i_{\pq}})=
(\sfd''_{\sfM}(\omegaup_{j_{1}}{\wedge}\cdots{\wedge}\omegaup_{j_{\qq}}))\wedge
\etaup_{i_{1}}{\wedge}\cdots{\wedge}\etaup_{i_{\pq}}.
\end{gather*}
These formulas yield explicit local expressions for $\sfd''_{\sfM}.$ We note that for $\pq{=}0$
they only depend on the choice of the coframe $\omegaup_{1},\hdots,\omegaup_{n}$ of
$\Et_{\sfM}.$ \par\smallskip
In particular, if $\sfM$ is compact and $\{\sfM_{\ttt}\}_{|\ttt|{<}\epi_{0}}$ a smooth deformation of
its $CR$ structure, then $\Et_{\sfM}$ stays transversal to $\Ta^{\ast}{}^{1,0}_{\!\!\sfM_{\ttt}}$
for small $|\ttt|$ and therefore the tangential Cauchy-Riemann complex in degree $0$ for $\sfM_{\ttt}$
can be defined on  sections of the same vector bundles for all small $\ttt.$ \par 

\subsection{$CR$ line bundles}\label{sub3.2a}
A \emph{$CR$ line bundle} on $\sfM$ 
is the datum of a smooth complex line bundle $\piup\,{:}\,\Ft{\to}\sfM,$ together
with a $CR$ structure on $\Ft$ which is compatible with the linear structure of
the fibers and for which $\piup$ is a $CR$ submersion. The 
compatible $CR$ structures on $\Ft$ are parametrized (modulo $CR$ equivalence)
by special cohomology 
classes of $H^{1}(\sfM,\Hi_{\sfM}^{0,1}),$ where $\Hi_{\sfM}^{0,1}$ is the sheaf
of germs of local cohomology classes of degree $(0,1)$ on $\sfM$: 
In a local trivialization $\{U_{i},s_{i})\}$ of $\Ft,$ with transition functions $(\gammaup_{i,j}),$
the structure is described by
the datum of $\bar{\partial}_{\sfM}$-closed forms $\phiup_{i}$ in 
$\Qq^{0,1}_{\,\sfM}(U_{i})$ 
with 
\begin{equation*}
 \bar{\partial}_{\sfM}^{\Ft}s_{i}=\phiup_{i} \otimes{s}_{i}\;\;\text{on $U_{i}$\; and}\;\;\;
 \phiup_{i}-\phiup_{j}=\gammaup_{i,j}^{-1}\bar{\partial}_{\sfM}\gammaup_{i,j}\;\:
 \text{on $U_{i}\cap{U}_{j}$.}
\end{equation*}
We can consistently define a differential  
operator $\bar{\partial}^{\Ft}_{\sfM}$ on smooth sections
of $\Ft,$ with values in the smooth sections of $\Qt_{\,\sfM}^{0,1}\otimes\Ft,$
in 
such a way that\footnote{This is obtained by identifying $\Ft$ with its vertical bundle and first defining 
$\bar{\partial}^{\Ft}_{\sfM}$ on the smooth local sections
$s$ of $\Ft$ by setting 
\begin{equation*}\tag{$\dag$}
 \bar{\partial}^{\Ft}_{\sfM}s(X{+}iJ_{\sfM}X)=ds(X){+}J_{\Ft}ds(J_{\sfM}X),\;\;\forall 
 X\,{\in}\, HM.
\end{equation*}
In fact, when $s$ is a section, the right hand side of ($\dag$) is vertical. }
\begin{equation*}
 \bar{\partial}_{\sfM}^{\Ft}(fs_{i})=(\bar{\partial}_{\sfM}f\otimes{s}_{i})+(f\,{\cdot}\, \phiup_{i})\otimes{s}_{i},\;\;
 \forall f\in\Ci(U_{i}).
\end{equation*}
\par 
Let us use the notation $\Fq_{\sfM}$ to indicate 
the sheaf of germs of smooth sections of $\Ft_{\sfM}$ and
$\Fq^{\pq,\qq}_{\sfM}$ for the sheaf of germs 
of smooth sections
on $\sfM$ of the bundle $\Qt^{\pq,\qq}_{\,\sfM}\,{\otimes}\,\Ft.$ 
Then we can define, in general, a differential operator 
$ \bar{\partial}_{\sfM}^{\Ft}$ mapping sections of $\Fq^{\pq,\qq}_{\sfM}$
into sections of  $\Fq^{\pq,\qq+1}_{\sfM}$ 
by setting 
for $U$ open in $\sfM,$ 
$\alphaup\,{\in}\,\Qq_{\,\sfM}^{p,q}(U)$ and $\sigma\,{\in}\,\Fq_{\sfM}(U),$ 
\begin{equation*}
 \bar{\partial}^{\Ft}_{\sfM}(\alphaup\otimes\sigma)=(\bar{\partial}_{\sfM}\alphaup)\otimes\sigmaup+
 (-1)^{p+q}\alphaup\wedge\partial^{\Ft}_{\sfM}\sigmaup.
\end{equation*}
In this way we obtain 
 the \emph{tangential Cauchy-Riemann complexes with coefficients in~$\Ft$\,}: 
\begin{equation*} \vspace{2pt}
\begin{aligned} 
\begin{CD}
 0 @>>> \Fq^{\pq,0}_{\,\sfM}(U) @>{\bar{\partial}_{\sfM}^{\Ft}}>> \Fq^{\pq,1}_{\,\sfM}(U) @>>> \cdots
\end{CD}\qquad\qquad \qquad  \\
 \qquad \qquad \qquad \qquad \begin{CD}
\cdots 
@>>> \Fq^{\pq,n-1}_{\,\sfM}(U) @>{\bar{\partial}_{\sfM}^{\Ft}}>> \Fq^{\pq,n}_{\,\sfM}(U) @>>> 0.
\end{CD}
\end{aligned}
\end{equation*}
We denote its cohomology groups by \vspace{2pt}
\begin{equation*}\vspace{2pt}
 H^{\pq,\qq}(U,\Ft_{\sfM})=\dfrac{\ker\big(\bar{\partial}_{\sfM}^{\Ft}:\Fq_{\,\sfM}^{\pq,\qq}(U) \to 
 \Fq_{\,\sfM}^{\pq,\qq+1}(U)\big)}{\text{image}\big(\bar{\partial}_{\sfM}^{\Ft}:\Fq_{\,\sfM}^{\pq,\qq-1}(U) \to 
 \Fq_{\,\sfM}^{\pq,\qq}(U)\big)}.
\end{equation*}

\par

When the $CR$ structure on $\sfM$ is induced from a generic 
embedding into an 
$(n{+}k)$-dimensional complex manifold $\sfX$, the $CR$ line bundle $\Kt_{\sfM}$ is
the restriction to $\sfM$ 
of the holomorphic canonical line bundle $\Kt_{\sfX}$ of $\sfX.$ 
In general, the canonical bundle of
an abstract $CR$ manifold  
 might not even be locally $CR$ trivializable,
as explained e.g. in \cite{BH, HN,  K}.
\par
For a $CR$ submanifold $\sfM$ of an $\Nt$-dimensional projective space $\CP^{\Nt},$ 
the restrictions $\cO_{\sfM}(-1)$ and $\cO_{\sfM}(1)$ to $\sfM$ of its tautological
and hyperplane bundles are $CR$ line bundles on $\sfM,$ dual to each other.
By the Euler sequence, 
the canonical bundle $\Kt_{\CP^{\Nt}}$ of the projective space is 
isomorphic to the $(\Nt{+}1)$-th power 
$\cO_{\CP^{\Nt}}(-\Nt{-}1)$ of the tautological bundle
(see e.g. \cite[p.146]{GH}) and therefore its dual $\Kt^{\ast}_{\CP^{\Nt}},$ being isomorphic
to the $(\Nt{+}1)$-th power $\cO_{\CP^{\Nt}}(\Nt{+}1)$
of  the hyperplane 
bundle $\cO_{\CP^{\Nt}}(1),$ which is very ample,
is very ample. 
 In particular, if $\sfM$ is a generic $CR$ submanifold of $\CP^{\Nt},$ then 
its anti-canonical bundle
 $\Kt^{\ast}_{\sfM}$ 
 is the restriction to $\sfM$ of a very ample line bundle. If $\Et_{\Xs}$ is a very ample line bundle
 on a complex compact complex manifold $\Xs,$ 
then its first Chern class $\cq_{1}(\Et_{\Xs})$ is positive
and the Kodaira-Nakano theorem  tells us that this property is sufficient to ensure that
the cohomology groups
$H^{\qq}(\Xs,\Kt_{\Xs}{\otimes}\Et_{\Xs})$ vanish for $\qq>0$ (see e.g. \cite[p.154]{GH}).
On the other hand, it is known that also the anti-canonical bundles of the generalized
complex flag manifolds are positive (see e.g. \cite{AKQ, Bo-Hi}). 
They are therefore examples of the larger class of \emph{Fano varieties},
which are characterized by the ampleness of their anti-canonical bundle $\Kt_{\Xs}^{*}.$ 
\par
When $\sfX$ is a  flag manifold of a complex semisimple Lie group $\Gf,$ 
the minimal orbits in $\sfX$ of its real forms $\Gf_{\R},$
studied e.g. in \cite{AMN,Wolf69},
are examples of generic $CR$ submanifolds of Fano varieties.

\subsection{Deformations and $CR$ line bundles}\label{sub3.2}
A Cartier divisor in $\sfX$ is the effective divisor of a section of a line bundle $\Ft_{\sfX}$
on $\sfX.$ A positive $\Ft_{\sfX}$ has a power which is very ample and therefore
the support of its Cartier divisor is the support 
of the pole divisor of a global meromorphic
function $f$ 
on $\sfX.$ \par 
Let $\sfX$ be a complex space, 
 $f$ a global meromorphic function on $\sfX.$ The corresponding 
 complex line bundle  
 $\Ft_{\sfX}$ can be described in the following way: 
Let $\{U_{i}\}$ be an open covering of $\sfX$ 
by connected open subsets with the property  that, for each index $i,$ 
there are relatively prime
$f_{i}',f_{i}''\in\cO_{\sfX}(U_{i})$ 
such that $f'_{i}/f''_{i}{=}f$ on $U_{i}.$ The quotients $f'_{i}/f'_{j}{=}f''_{i}/f''_{j}=g_{i,j}$ define
nowhere vanishing holomorphic functions on $U_{i}\,{\cap}\,U_{j},$ 
which are the transition functions of 
a holomorphic line bundle $\Ft_{\sfX}$ 
on the trivialization atlas $\{U_{i}\}.$ 
Then
\begin{equation}\label{e3.2}
 D={\bigcup}_{i}\{\pct\in{U}_{i}\mid f''_{i}(\pct)=0\}
\end{equation}
is the support of the pole divisor  of $f.$ 
\par \smallskip
\begin{prop}\label{p3.2}
 Let $\sfX$ be a complex space, $D$ the 
 support of the pole divisor of 
 a global meromorphic function $f$ on $\sfX$ and $\Ft_{\sfX}$ the corresponding
 line bundle. Let $\sfM$ be a smooth $CR$ submanifold of $\sfX$ 
 and $\{\sfM_{\ttt}\}$ a deformation on $\sfM$ of the $CR$ structure $\sfM_{0}$
 induced on $\sfM$ by $\sfX$ 
 such that for each $\ttt$ 
 the $CR$ structures of $\sfM_{\ttt}$ and
 $\sfM_{0}$ agree to infinite order on $D\,{\cap}\,\sfM.$ 
 Assume that the $CR$ embedding $\sfM_{0}\,{\hookrightarrow}\,\sfX$ is generic.
 Then we can define $CR$ line bundles $\Ft_{\sfM_{\ttt}}$ on each $\sfM_{\ttt}$
 in such a way that $\{\Ft_{\sfM_{\ttt}}\}$ is a smooth deformation of
 $\Ft_{\sfM},$ and their  $CR$ structures agree to infinite order on the fibers
 over $D.$  
\end{prop}\begin{proof}
We keep the notation introduced at the beginning of the subsection.
The restrictions of the transition
functions $g_{i,j}$ define a $CR$ line bundle on ${M_{0}},$ because 
\begin{equation*}
 g_{i,j}^{-1}\bar{\partial}_{M}g_{i,j}=0 \;\;\text{on $U_{i}\cap{U}_{j}\cap\sfM.$}
\end{equation*}
Let us consider now an element of the
 deformation $\{\sfM_{\ttt}\}.$ 
We note that 
\begin{equation*}
g_{j,i} \bar{\partial}_{\sfM_{\ttt}}g_{i,j}=\frac{f''_{j}}{f''_{i}}
\frac{f''_{j}\bar{\partial}_{\sfM_{\ttt}}f''_{i}-f''_{i}\bar{\partial}_{\sfM_{\ttt}}f''_{j}}{
f''_{j}{}^{2}} = \frac{\bar{\partial}_{\sfM_{\ttt}}f''_{i}}{f''_{i}}-\frac{\bar{\partial}_{\sfM_{\ttt}}f''_{j}}{f''_{j}},\quad
\text{on $U_{i}\cap{U}_{j}\cap\sfM$}
\end{equation*}
is the difference of two smooth $\bar{\partial}_{\sfM_{\ttt}}$-closed forms, 
the first one defined on $U_{i}{\cap}\,\sfM,$ the second on $U_{j}{\cap}\,\sfM.$
Indeed 
\begin{equation*}
 \phiup_{i}^{(\ttt)}= 
\begin{cases} f''_{i}{}^{-1}\cdot \bar{\partial}_{\sfM_{\ttt}}f''_{i},
 & \text{on $(U_{i}\cap\sfM){\backslash}D,$}\\
 0 , & \text{on $U_{i}\cap\sfM\cap{D},$}
\end{cases}
\end{equation*}
is a smooth $\bar{\partial}_{\sfM_{\ttt}}$-closed $(0,1)$ form on $U_{i}\,{\cap}\,\sfM,$ because
$\bar{\partial}_{\sfM_{\ttt}}f''_{i}$ vanishes to infinite order on $D\,{\cap}\,U_{i}{\cap}\,\sfM$ by the
assumption that the $CR$ structures of $\sfM$ and $\sfM_{\ttt}$ agree to infinite
order on $D\,{\cap}\,\sfM,$ while $f''_{i},$ which extends holomorphically 
to the neighourhood  $U_{i}$ of $U_{i}\,{\cap}\,\sfM$ in $\sfX,$ 
cannot vanish to infinite order at any point of
$D\,{\cap}\,\sfM.$  \par 
The complex smooth line bundle $\Ft_{\sfM}$
is obtained by gluing  the trivial line bundles $\{U_{i}\,{\times}\,\C\}$ by the
transition functions $\{g_{i,j}\}.$ Let $\sigmaup_{i}$ be the section $U_{i}\,{\times}\,\{1\}$
on $U_{i}.$ The $CR$ structure on $\Ft_{\sfM_{\ttt}}$ is obtained by setting 
\begin{equation}\label{e3.3}
 \bar{\partial}^{\Ft}_{\sfM_{\ttt}}\sigmaup_{i}=\phiup^{(\ttt)}_{i}{\cdot}\sigmaup_{i}\;\;\;\text{on $U_{i}$
 for all $i.$}
\end{equation}
Clearly it agrees to infinite order on $D$ with that of $\Ft_{\sfM_{0}}.$
\end{proof}

\section{An example}\label{secex}
Let $m$ be an integer $\geq{2}$ and 
$\sfM$ the quadric hypersurface of $\CP^{m}$ described in homogeneous coordinates by 
\begin{equation*}
 \sfM=\{z_{0}\bar{z}_{0}+z_{1}\bar{z}_{1}=z_{2}\bar{z}_{2}+\cdots+z_{m}\bar{z}_{m}\}.
\end{equation*}
The rational function $\zetaup_{1}\,{=}\,z_{1}/z_{0}$ has divisor $(\zetaup_{1})=D_{1}-D_{0}$ with
\begin{equation*}
 D_{0}=\{z_{0}=0\},\;\;\; D_{1}=\{z_{1}=0\}\;\;\;\text{and} \;\; D_{1}\cap{D}_{2}\cap\sfM=\emptyset.
\end{equation*}
Since the point $[1,0,\hdots,0]$ does not belong to $\sfM,$ 
the projection map 
\begin{equation*} \piup:
\CP^{m}{\backslash} \{[1,0,\hdots,0]\}\ni 
 [z_{0},z_{1},\hdots,z_{m}]\longrightarrow [z_{1},\hdots,z_{m}]\in\CP^{m-1}
\end{equation*}
is well defined on $\sfM.$ Its image is the complement in $\CP^{m-1}$
of the open ball 
$$\Bs=\{z_{2}\bar{z}_{2}{+}\cdots{+}z_{m}\bar{z}_{m}\,{<}\,z_{1}\bar{z}_{1}\}.$$ 
 \par
Let $F$ be a smooth function on $\CP^{m-1}$ which is holomorphic on $\Bs$
and $\qct$ any point of its boundary $\partial\Bs.$
If we can find an open neighbourhood $V$ of $\qct$ in $\CP^{m-1}$ and a solution 
$\vq$ to the Cauchy problem 
\begin{equation*} 
\begin{cases}
\vq\in\Ci(V{\backslash}\bar{\Bs})\cap\Cc^{0}(V{\backslash}\Bs),\\
 \bar{\partial}\vq=\bar\partial{F}, \qquad \text{on $V{\backslash}\bar{\Bs},$}\\
 \vq=0, \quad\;\qquad \text{on $V\cap\partial\Bs,$}
\end{cases}
\end{equation*}
then 
\begin{equation*}
 F'= 
\begin{cases}
 F, & \text{on $V\cap\Bs,$}\\
 F-\vq, & \text{on $V{\backslash}\Bs$}
\end{cases}
\end{equation*}
continues analytically $F|_{\Bs}$ to the point $\qct.$\par  
There are holomorphic functions
on $\Bs$ which can be continued smoothly, but not analytically, to all points of $\partial\Bs.$
Thus we can find $F\,{\in}\,\Ci(\CP^{m-1}),$ holomorphic on $\Bs,$ 
such that $F|_{\Bs}$ has no analytic continuation to any point of $\partial\Bs$ and moreover,
since $\partial\Bs\,{\cap}\,\{z_{1}=0\}\,{=}\,\emptyset,$ 
we can require that its support does not
intersect the hyperplane $\{z_{1}\,{=}\,0\}.$
Then $\alphaup=\bar\partial{F}$ is a smooth $(0,1)$ form on $\CP^{m-1}$ with the properties:
\begin{itemize}
\item $\bar\partial\alphaup=0$ and $\supp(\alphaup)\subset\CP^{m-1}{\backslash}\Bs,$ 
$\supp(\alphaup)\cap\{z_{1}=0\}=\emptyset$;
 \item If   
 $V$ is any open neighourhood
 of a point $\qct\,{\in}\,\partial\Bs$ and \par
 \centerline{
 $\vq\,{\in}\,\Ci(V{\backslash}\bar{\Bs})\,{\cap}\,\Cc^{0}(V{\backslash}\Bs)$ 
 satisfies $\bar{\partial}\vq\,{=}\,\alphaup$ on $V{\backslash}\bar{\Bs},$ }\par\noindent
then $\vq$ is not $0$ at some point of $V\,{\cap}\,\partial\Bs.$  
\end{itemize}
\par 
Let us fix a $(0,1)$ form $\alphaup$ on $\CP^{m-1}$ with these properties. 
\par \smallskip
We consider the coordinate charts $(U_{0},\zetaup),$ $(U_{1},\xiup)$ 
in $\CP^{m}$ with 
\begin{align*}
 U_{0}=\{z_{0}\neq{0}\}=\CP^{m}{\backslash}D_{0},\;\; \zetaup_{1}=\frac{z_{1}}{z_{0}},\;\zetaup_{2}=
\frac{z_{1}}{z_{0}},\hdots , \; \zetaup_{m}=\frac{z_{m}}{z_{0}},\\
U_{1}=\{z_{1}\neq{0}\}=\CP^{m}{\backslash}D_{1},\;\; \xiup_{1}=\frac{z_{0}}{z_{1}},\;\xiup_{2}=
\frac{z_{2}}{z_{1}},\hdots , \; \xiup_{m}=\frac{z_{m}}{z_{1}}.
\end{align*}
\par 
The pullback $\piup^{*}\alphaup$ vanishes to infinte order
on the points of $\sfM\,{\cap}\,D_{0},$ which are mapped by $\piup$ onto
the boundary of $\Bs.$ Then $\zetaup_{1}\,{\cdot}\,\piup^{*}\alphaup,$
defined on $\sfM{\backslash}D_{0},$ extends to a 
smooth $(0,1)$ form $\betaup$ on $\sfM,$ 
vanishing to infinite order on $\sfM\,{\cap}\,D_{0},$ 
whose support is contained in $U_{1}\,{\cap}\,\sfM.$ 
\par
We can define a smooth $CR$ structure $\sfM_{\ttt}$ on $\sfM$ by requiring that
its canonical bundle has sections 
\begin{equation*} 
\begin{cases} \sq_{\ttt,0}=
 \sfd\zetaup_{1}\wedge\sfd\zetaup_{2}\wedge\cdots\wedge\sfd \zetaup_{m}, & \text{on $U_{0}\cap\sfM
 {\backslash}\supp(\betaup)$,}\\
 \sq_{\ttt,1}=
 (\sfd \xiup_{1}+\ttt\betaup) \wedge\sfd\xiup_{2}\wedge\cdots\wedge\sfd\xiup_{m}, &
 \text{on $U_{1}\cap\sfM.$}
\end{cases}
\end{equation*}
\par
Note that 
\begin{align*}
 \sfd\sq_{\ttt,1}&=\ttt \left( -\frac{\sfd\xiup_{1}}{\xiup_{1}^{2}}\wedge\piup^{*}(\alphaup)
 +\frac{1}{\xiup}\piup^{*}(\sfd\alphaup)
 \right) \wedge\sfd\xiup_{2}\wedge\cdots\wedge\sfd\xiup_{m}
 \\ &=
 -\ttt\frac{\sfd\xiup_{1}}{\xiup_{1}}\wedge\betaup\wedge\sfd\xiup_{2}\wedge\cdots\wedge\sfd\xiup_{m}
 = \dfrac{\ttt\;}{\xiup_{1}}\betaup\wedge\sq_{\ttt,1}
\end{align*}
because  
\begin{equation*}\tag{$\dag$}\left\{
\begin{aligned}
 \piup^{*}(\sfd\alphaup)\wedge\sfd\xiup_{2}\wedge\cdots\wedge\sfd\xiup_{m}=
  \piup^{*}(\sfd\alphaup\wedge\sfd\xiup_{2}\wedge\cdots\wedge\sfd\xiup_{m})\qquad \\
  =
   \piup^{*}((\bar\partial\alphaup)\wedge\sfd\xiup_{2}\wedge\cdots\wedge\sfd\xiup_{m})=0.
\end{aligned}\right.
\end{equation*}
Note that  the form $\betaup/\xiup_{1}\,{=}\,\piup^{*}\alphaup/\xiup_{1}^{2}$ 
extends to a smooth form $\betaup',$ defined on $\sfM$ and
vanishing to infinite order on $D_{0},$ so that 
\begin{equation*} \tag{$\ddag$}
\sfd\sq_{\ttt,0}=0\,\;\text{on $U_{0}\cap\sfM{\backslash}\supp(\betaup)$}\;\;\text{and}
\;\; 
 \sfd\sq_{\ttt,1}=\ttt\,{\cdot}\,\betaup'\wedge\sq_{\ttt,1}\;\;\text{on $U_{1}\cap\sfM$}
\end{equation*}
 shows that the line bundles 
 $\Kt_{\sfM_{\ttt}}$ 
 define
formally integrable complex valued distributions.\par

 \par \smallskip
These  $CR$ structures agree to infinite order on $\sfM{\cap}D_{0}$ 
with the $CR$ structure $\sfM_{0}$ on $\sfM$ which is induced by its embedding into $\CP^{m}$ 

Their canonical bundles are isomorphic  
complex line bundles,
since the transition functions are in all cases $g_{0,1}{=}\,\sq_{\ttt,0}{/}\sq_{\ttt,1}\,{=}\,
{-}\zetaup_{1}^{m+1}\,{=}\,{-}\xiup_{1}^{-m-1},$ $g_{1,0}{=}g_{0,1}^{-1}$ 
on $\sfM{\backslash}(D_{0}{\cup}D_{1}{\cup}\supp(\betaup)).$ \par
We note that $\xiup_{2},\hdots,\xiup_{m}$ are $CR$ functions on $\sfM_{\ttt}{\backslash}{D}_{0}$
for all $\ttt.$ A necessary  condition for $\sfM_{\ttt}$ being locally
embeddable at a point $\pct$ of $D_{0}$ is that one could find an open neighbourhood 
$U'$ of $\pct$ in $U_{1}\,{\cap}\,\sfM$ and a function $\uq\,{\in}\,\Ci(U')$ such that 
\begin{equation*}
 \sfd(e^{\uq}\sq_{\ttt,1})=0 \;\;\text{on $U'$}. 
\end{equation*}
This equation can be rewritten, by factoring out $e^{\uq}\zetaup_{1}$ 
and taking into account $(\dag),$ as
\begin{equation*}\tag{$*$}
 \left(\xiup_{1}\sfd\uq +\dfrac{\ttt\;}{\xiup_{1}}\piup^{*}\alphaup\right)
  \wedge\sfd\xiup_{2}\wedge\cdots\wedge\sfd\xiup_{m}=0\;\; \text{on $U'{\backslash}{D}_{0}.$}
\end{equation*}

The fibers of $\piup$ above points of $\CP^{m-1}{\backslash}\bar{\Bs}$ are circles and 
reduce to single points over $\partial\Bs.$  Since the projection $\piup$ is proper, 
by shrinking we may assume that $\piup^{-1}(\piup(U'))=U'.$ 
By integrating over the fibers we define 
\begin{equation*}
 \vq(\qct)=\frac{-1}{2\pi{i}}\oint_{\piup^{-1}(\qct)}\uq\, \xiup_{1} d\xiup_{1},\;\;\text{for $\qct\in\piup(U'){\backslash}
 \partial\Bs.$} 
\end{equation*}
This is a smooth function on $\piup(U'){\backslash}
 \partial\Bs,$ which extends to a continuous function, that we still denote by $\vq,$
 vanishing on $\piup(U')\cap\partial\Bs.$ 
Then from $(*)$ we obtain that 
\begin{equation*}
 \sfd\vq\wedge\sfd\xiup_{2}\wedge\cdots\wedge\sfd\xiup_{m}=
 \ttt\,\alphaup\wedge\sfd\xiup_{2}\wedge\cdots\wedge\sfd\xiup_{m},\;\;\;\text{on $\piup(U'){\backslash}\bar{\Bs}.$}
\end{equation*}
This is equivalent to 
\begin{equation*}
 \bar\partial\vq = \ttt\,\alphaup,\;\;\; \text{on $\piup(U'){\backslash}\bar{\Bs}$}
\end{equation*}
and hence contradicts the choice of $\alphaup$ if $\ttt{\neq}0,$ 
proving that $\sfM_{\ttt}$ cannot be locally embedded at any point $\pct\,{\in}\,D_{0}$
if $\ttt{\neq}0.$ 
\par 
\smallskip
This example from \cite{HN2} shows that we can construct 
a smooth deformation   
$\{\sfM_{\ttt}\}$ of the $CR$ structure $\sfM_{0}$ of the quadric $\sfM$ such that 
no $\sfM_{\ttt}$ 
can be locally embedded into a complex manifold at any point of $D_{0}$ 
if $\ttt{\neq}0.$ \par
The $CR$ structure of the $\sfM_{\ttt}$'s are of type $(m{-}1,1)$ and $\sfM_{0}$
is $1$-pseu\-do\-con\-cave. Hence the example
shows that Theorem~\ref{main} cannot be valid if we only require
$\sfM_{0}$ being $1$-pseudoconcave. 
\par\smallskip

\begin{rmk}
Equations $(\ddag)$ is of the type of those we need to solve to construct 
projective embeddings in \S\ref{sec6}. It 
can be rewritten as 
\begin{equation*} 
\begin{cases}
 \bar{\partial}^{\,\Kt}_{\sfM_{\ttt}}\sq_{\ttt,0}=0,\\ 
 \bar{\partial}^{\,\Kt}_{\sfM_{\ttt}}\sq_{\ttt,1}=
 \ttt{\cdot}\betaup'\,\sq_{\ttt,1}.
\end{cases}
\end{equation*}
We also have 
\begin{equation*}
 \bar{\partial}_{\sfM_{\ttt}}\uq = \bar{\partial}_{\sfM_{0}}\uq 
 -\ttt\,{\cdot}\,\frac{\partial{u}}{\partial\xiup_{1}}\,\betaup.
\end{equation*}
Then we obtain 
\begin{equation*}
 \bar{\partial}^{\,\Kt}_{\sfM_{\ttt}}\left(\frac{1}{\xiup_{1}}\sq_{\ttt,1}^{-1}\right)=
 -\ttt\left( \frac{1}{\xiup_{1}^{2}} \betaup+\frac{1}{\xiup_{1}}\betaup'\right)\,\sq_{\ttt,1}^{-1}
 =-2\,\frac{\ttt\;}{\xiup_{1}^{3}}
 \piup^{*}\alphaup\,\sq_{\ttt,1}^{-1}.
\end{equation*}
An argument similar to the one used above shows that the equation  
\begin{equation*}\tag{$**$}
 \bar{\partial}(\wq \cdot \sq_{\ttt,1}^{-1})= (\bar{\partial}_{\sfM_{0}}\wq-\ttt\frac{\partial\wq}{\partial\xiup_{1}}\betaup
-\ttt\cdot\wq \betaup') \sq_{\ttt,1}^{-1}=-2\,\frac{\ttt\;}{\xiup_{1}^{3}}
 \piup^{*}\alphaup\,\sq_{\ttt,1}^{-1}
\end{equation*}
has no smooth solution on $\sfM.$ 
Indeed  $(**)$ can be rewritten in the form 
\begin{equation*}
\left(\sfd\wq-\frac{\ttt}{\xiup_{1}^{2}}\frac{\partial(\xiup_{1}\wq)}{\partial\xiup_{1}}\piup^{*}\alphaup+
\frac{2\ttt}{\xiup_{1}^{3}}\piup^{*}\alphaup\right)\sfd\xiup_{2}\wedge\cdots\wedge\sfd\xiup_{m}=0
\end{equation*}
Multiplying by $\xiup^{2}_{1}$ and integrating on the fiber, we see that the term in the middle
vanishes, while 
\begin{equation*}
 \vq(\qct)=\frac{-1}{4\pi{i}}\oint_{\piup^{-1}(\qct)} (\xiup_{1}^{2}\wq)\sfd\xiup_{1}
\end{equation*}
 yields a solution in $\Ci(\CP^{m-1}{\backslash}\bar{\Bs})\,{\cap}\,\Cc^{0}(\CP^{m-1}{\backslash}\Bs),$
 vanishing on $\partial\Bs,$ of 
 the equation $\bar{\partial}\vq\,{=}\,\alphaup$ on $\CP^{m-1}{\backslash}\bar{\Bs},$
 contradicting the choice of $\alphaup.$ \par
 This adds a further argument for the need of conditions stronger than 1-pseudoconcavity
 for the validity of the conclusions of Theorem~\ref{main}.
\end{rmk}

\section{Vanishing results}\label{sec5}
We keep the notation introduced in \S\ref{secnot}, \S\ref{secprel}. 
\subsection{Embedded $CR$ manifolds}
We have the following vanishing result:
\begin{thm} \label{vanish}
Assume that $\sfM$ is a 2-pseudoconcave generic smooth 
$CR$ submanifold of type $(n,k)$ of a 
complex variety $\Xs.$ Let $\Ft_{\Xs}$ be a positive line bundle on $\Xs$ and $\Ft_{\sfM}$ its
restriction to $\sfM.$ Then we can find a positive integer $\ell_{0}$ such that 
\begin{equation}
 H^{\pq,\qq}(\sfM,\Ft^{\,-\ell}_{\!\!\sfM})=0,\quad \forall \ell{\geq}\ell_{0},\; 0{\leq}\pq{\leq}n{+}k,\;\qq=0,1.
\end{equation}
 \end{thm}
\begin{proof}
We know from \cite{HN1} that a generic $2$-pseudoconcave 
$\sfM$ 
has a fundamental system of
tubular neighborhoods $U$ in $\Xs$ that are 
$(n{-}2)$-pseu\-do\-con\-vex 
and $(n{+}k{-}2)$-pseu\-do\-con\-cave in the sense of 
Andreotti-Grauert\footnote{A real valued smooth function $\phiup$ on an $\Nt$-dimensional
complex manifold $\Xs$ is strongly $\qq$-pseudoconvex in the sense of \cite{AG} 
at points where its complex Hessian has at least $(\Nt{-}\qq{+}1)$ positive eigenvalues. 
Then $\Xs$ is called
\emph{strictly $\qq$-pseudoconvex} if there is an exhaustion function 
$\phiup\,{\in}\,\Ci(\Xs,\R)$ which is 
strictly $\qq$-pseudoconvex outside a compact subset of $\Xs$
and 
strictly $\qq$-pseudoconcave 
if there is an exhaustion function $\phiup\,{\in}\,\Ci(\Xs,\R)$ such that $({-}\phiup)$ is
strictly $\qq$-pseudoconvex outside a compact subset of $\Xs.$} 
and for which, for all
$0{\leq}\pq{\leq}n{+}k,$  
the natural restriction maps
$$H^{\qq}(U, \Omega^{\pq}(\Ft_{U})) \longrightarrow H^{\qq}(M, \Omega^{\pq}(\Ft_{\sfM})),\;\;
\qq=0,1,$$ are 
isomorphisms for any holomorphic line bundle $\Ft_{U}$ over $U$.  
\par 
This was proved in \cite[Thm.4.1]{HN1} for the trivial line bundle. 
The arguments of the proof given there
apply also in general,  yielding the isomorphism for
arbitrary holomorphic vector 
bundles.
The statement is then a consequence of 
 \cite[Prop.27]{AG}, because $U$ is 
 \textit{strongly $(n{+}k{-}2)$-concave} in the sense of Andreotti-Grauert and hence 
 also $\Ft^{-1}_{U}\,{\coloneqq}\,\left.\Ft^{-1}_{\Xs}\right|_{U}$ is 
 strongly $(n{+}k{-}2)$-concave as a complex space, according to \cite{AG}
 and therefore there is a positive integer $\ell_{0}$ such that 
 $H^{\pq,\qq}(U,\Ft^{\,-\ell}_{\Xs})\,{=}\,0$ for $\ell{\geq}\ell_{0},$ 
 $0{\leq}\pq{\leq}n{+}k$ and $\qq\,{=}\,0,1.$ \end{proof}
 \subsection{Some estimates for abstract $CR$ manifolds}\label{sec5.2}
 Let $\sfM$ be an abstract compact 
 $CR$ manifold of type $(n,k)$ and $\Ft_{\sfM}$ a complex $CR$ line
 bundle on $\sfM.$ 
 To discuss its cohomology groups $H^{\pq,\qq}(\sfM,\Ft_{\sfM})$ 
 by using functional analytic methods, it is convenient to fix 
a smooth Riemannian metric on $\sfM$ and a 
smooth Hermitian norm on the fibers of $\Ft_{\sfM}.$  
In this way we are able,   for each real $r{\geq}0,$ 
to define in  a standard way
on the smooth sections of $\Qt_{\sfM}^{\pq,\qq}\,{\otimes}\,\Ft_{\sfM}$  
 the Sobolev 
 $\sfW^{r}$-norm, involving the $L^{2}$-norms of 
 their fractional derivatives up to order $r.$ 
 The completions of $\Fq^{\pq,\qq}(\sfM)$ 
 with respect to these norms are Hilbert spaces,
 that we denote by $\sfW^{r}(\sfM,\Qt_{\sfM}^{\pq,\qq}{\otimes}\Ft_{\sfM}).$ 
We use the notation
\begin{equation*}
 (\vq_{1}|\vq_{2})_{r}\;\;\text{and}\;\; \|\vq\|_{r}=\sqrt{(\vq|\vq)_{r}},\;\;\;\text{for
 $\vq_{1},\vq_{2},\vq\in\sfW^{r}(\sfM,\Qt_{\sfM}^{\pq,\qq}\otimes\Ft_{\sfM})$}
\end{equation*}
for their Hermitian scalar products and norms (see e.g. \cite{He}).

 In  \cite{HN1} subelliptic estimates were proved for the $\bar{\partial}_{\sfM}$
 complex under pseudoconcavity conditions. Since the arguments are local,
 they trivially extend to forms with coefficients in a 
complex $CR$ line bundle. We have in particular  
the following statement.
\begin{prop}\label{p5.2}
 Let $\sfM$ be a compact  smooth 
  abstract 
$CR$ manifold of type $(n,k)$ and
 $\Ft_{\sfM}$ any complex $CR$ line bundle on $\sfM.$ Assume that
 $\sfM$ is $2$-pseu\-do\-con\-cave. Then, for every 
  $0{\leq}\pq{\leq}n{+}k$ and every  
 real $r{\geq}0$ there
 is a positive constant $c_{r,\pq}{>}0$ such that 
\begin{equation}\vspace{-20.5pt} \label{e5.2} \begin{cases}
c_{r,\pq}\cdot\|\uq\|_{r+(1/2)} \leq\|\bar{\partial}^{\Ft}_{\sfM}\uq\|^{2}_{r},\quad\forall 
\uq\in 
 \Fq^{\pq,0}(\sfM),\\
 c_{r,\pq}\cdot\|\vq\|_{r+(1/2)}\leq \|\bar{\partial}^{\Ft}_{\sfM}\vq\|^{2}_{r}
 + \|(\bar{\partial}^{\Ft}_{\sfM})^{\ast}\vq\|^{2}_{r}+\|\vq\|_{0}^{2},
 \;\; \forall \vq\in 
 \Fq^{\pq,1}(\sfM).
 \end{cases}
\end{equation}\qed
\end{prop}
\vspace{4pt}
Let us define the subspace 
 \begin{equation}\label{e5.3a}
 \cN_{\Ft}^{\pq,1}(\sfM)=\{\vq\in
 \Fq^{\pq,1}(\sfM)
 \mid \bar{\partial}^{\Ft}_{\sfM}\vq=0,\;\;
 (\bar{\partial}^{\Ft}_{\sfM})^{\ast}\vq=0\}.
\end{equation}
 \par
Keeping the assumptions and the notation above, we have
 \begin{prop} \label{p5.3}
 Under the assumptions of Proposition~\ref{p5.2}: \begin{itemize}
 \item The space $\cN_{\Ft}^{\pq,1}(\sfM)$ is finite dimensional
  and equals its weak closure with respect to the $L^{2}$ norm; 
 \item For every pair of real numbers $r,r'$ with $-(1/2){\leq}r'{<}r,$ $r{\geq}0,$ there is
 a constant $C_{r,r',\pq}{>}0$ such that 
 \begin{equation}\label{e5.3a} \left\{\begin{aligned}
 \|\vq\|_{r'+(1/2)}\leq C_{r,r'\pq}\left(\|\bar{\partial}^{\Ft}_{\sfM}\vq\|^{2}_{r}
 {+} \|(\bar{\partial}^{\Ft}_{\sfM})^{\ast}\vq\|^{2}_{r}\right),
 \;\; \qquad \qquad\qquad\\
 \forall \vq\in 
 \Fq^{\pq,1}(\sfM){\cap}(\cN_{\Ft}^{\pq,1}(\sfM))^{\perp}.
 \end{aligned}\right.
\end{equation}\end{itemize} 
\noindent
[The orthogonal is taken with respect to the $L^{2}$-scalar product on~$\Fq^{\pq,1}(\sfM).$]
\begin{itemize}
 \item $H^{\pq,1}(\sfM,\Ft_{\sfM})\simeq\cN_{\Ft}^{\pq,1}(\sfM).$ 
\end{itemize}
 \end{prop} 
\begin{proof}
 The first statement is a consequence of \eqref{e5.2}: this subelliptic estimate for $r{=}0$ 
 implies in particular that
 the $L^{2}$ and $\sfW^{1/2}$-Sobolev norms are equivalent on $\cN_{\Ft}^{\pq,1}(\sfM).$
 Then by Rellich's  theorem the unit $L^{2}$ ball in $\cN_{\Ft}^{\pq,1}(\sfM)$ is compact
 and this implies that $\cN_{\Ft}^{\pq,1}(\sfM)$ is finite dimensional
  and thus also
  equals its weak closure with respect to the $L^{2}$ norm. \par\smallskip
  The second statement can be proved by contradiction: assume that there are  
  real numbers $r,r',$ with $-(1/2){\leq}r'{<}r,$ $r{\geq}0,$ for which \eqref{e5.3a}
  is not valid for any positive constant $C_{r,r',\pq}.$\par 
   Then we can find a sequence
  $\{\vq_{\nuup}\}\,{\subset}\,\Fq^{\pq,1}(\sfM){\cap}(\cN_{\Ft}^{\pq,1}(\sfM))^{\perp}$ with
  \begin{equation*}
 \|\vq_{\nuup}\|_{r'+(1/2)}=1,\quad \|\bar{\partial}^{\Ft}_{\sfM}\vq\|^{2}_{r}
 {+} \|(\bar{\partial}^{\Ft}_{\sfM})^{\ast}\vq\|^{2}_{r}<2^{-\nuup}.
\end{equation*}
By \eqref{e5.2} this sequence is bounded in the $\sfW^{r+(1/2)}$-norm and 
therefore has
a subsequence $\{\vq_{\nuup'}\}$ which, weakly in $\sfW^{r+(1/2)}$ and hence
strongly in $\sfW^{r'+(1/2)},$ converges to a $\vq_{\infty}.$ 
In particluar $\|\vq_{\infty}\|_{r'+(1/2)}\,{=}\,1$ shows that  $\vq_{\infty}{\neq}0.$ 
This limit is a weak
solution of $\bar{\partial}^{\Ft}_{\sfM}\vq_{\infty}\,{=}\,0$ and
$(\bar{\partial}^{\Ft}_{\sfM})^{*}\vq_{\infty}\,{=}\,0$ and therefore, by the hypoellipticity
of $\bar{\partial}^{\Ft}_{\sfM}\,{\oplus}\,(\bar{\partial}^{\Ft}_{\sfM})^{*}$  
following from the subellipticity estimate \eqref{e5.2}, is a smooth section and thus
a nonzero
element of $\cN_{\Ft}^{\pq,1}(\sfM).$ 
\par
This completes the proof of the second statement. \par\smallskip
It is clear that the elements of $\cN_{\Ft}^{\pq,1}(\sfM)$ represent cohomology classes
in $H^{\pq,1}(\sfM,\Ft_{\sfM}).$ Indeed, if $\vq\,{\in}\,\cN_{\Ft}^{\pq,1}$ and 
$\vq\,{=}\,\bar{\partial}^{\Ft}_{\sfM}\uq$ for some $\uq\,{\in}\,\Fq^{\pq,0}(\sfM),$
then 
\begin{equation*}
(\bar{\partial}^{\Ft}_{\sfM})^{*}\bar{\partial}^{\Ft}_{\sfM}\uq=0 \;\Longrightarrow\;
 \|\vq\|_{0}^{2}=(\vq|\bar{\partial}^{\Ft}_{\sfM}\uq)_{0}=
 ((\bar{\partial}^{\Ft}_{\sfM})^{*}\vq | \uq)_{0}=0
 \end{equation*}
 shows that $\vq\,{=}\,0.$ On the other hand, if $f\,{\in}\,\Fq^{\pq,1}(\sfM)$ and 
 $\bar{\partial}^{\Ft}_{\sfM}f\,{=}\,0,$ then we can decompose $f$ into the sum
 $f\,{=}\,f_{0}{+}f_{1},$ with $f_{0}\,{\in}\,\cN_{\Ft}^{\pq,1}(\sfM)$ and 
 $f_{1}\,{\in}\,(\cN_{\Ft}^{\pq,1}(\sfM))^{\perp}.$ \par We get 
\begin{align*}
 |(f_{1}|\vq)_{0}|\leq\|f_{1}\|_{0}\|\vq\|_{0}\leq C_{-1/2,0,\pq}(\|\bar\partial^{\Ft}_{\sfM}\vq\|^{2}
 +\|(\bar\partial^{\Ft}_{\sfM})^{*}\vq\|^{2})^{1/2},\;\;
 \qquad\quad\\
 \forall \vq\in
 \Fq^{\pq,1}(\sfM){\cap}(\cN_{\Ft}^{\pq,1}(\sfM))^{\perp}.
\end{align*}
By Riesz' representation theorem 
there is a unique  
$\wq\,{\in}\,\Fq^{\pq,1}(\sfM){\cap}(\cN_{\Ft}^{\pq,1}(\sfM))^{\perp}$
($\wq$ is smooth because of the subellipticity estimate \eqref{e5.2})
such that  
\begin{equation*}
 (\bar{\partial}^{\Ft}_{\sfM}\wq|\bar{\partial}^{\Ft}_{\sfM}\vq)_{0}+
 \,((\bar{\partial}^{\Ft}_{\sfM})^{*}\wq|(\bar{\partial}^{\Ft}_{\sfM})^{*}\vq)_{0}
 =(f_{1}|\vq),\; \forall \vq\in
 \Fq^{\pq,1}(\sfM){\cap}(\cN_{\Ft}^{\pq,1}(\sfM))^{\perp}.
\end{equation*}
Since $f_{1}$ is $L^{2}$-orthogonal to $\cN_{\Ft}^{\pq,1}(\sfM),$ this equality holds true
for all $\vq$ in $\Fq^{\pq,1}(\sfM)$ and integration by parts yields 
\begin{equation*}
 (\bar{\partial}^{\Ft}_{\sfM})^{*}\bar{\partial}^{\Ft}_{\sfM}\wq=0,\qquad 
 \bar{\partial}^{\Ft}_{\sfM}(\bar{\partial}^{\Ft}_{\sfM})^{*}\wq=f_{1}.
\end{equation*}
This shows that the orthogonal projection of $\ker(\bar{\partial}^{\Ft}_{\sfM}:\Fq^{\pq,1}(\sfM)
\to\Fq^{\pq,2}(\sfM))$ onto $\cN_{\Ft}^{\pq,1}(\sfM)$ yields, by passing to the 
injective quotient, an isomorphism between $H^{\pq,1}(\sfM,\Ft)$ and
$\cN^{\pq,1}_{\Ft}(\sfM).$ \par
The proof is complete.
\end{proof}
Likewise, we have a similar statement for $CR$ 
sections of $\Ft$ on $\sfM.$ Set 
\begin{equation*}
 \cN^{\pq,0}_{\Ft}(\sfM)=\{\uq\in\Fq^{\pq,0}(\sfM)\mid \bar\partial^{\Ft}_{\sfM}\uq=0\}.
\end{equation*}
 \begin{prop} \label{p5.3a}
 Under the assumptions of Proposition~\ref{p5.2}: \begin{itemize}
 \item The space $\cN^{\pq,0}_{\Ft}(\sfM)$ is finite dimensional
  and equals its weak closure with respect to the $L^{2}$ norm; 
 \item For every pair of real numbers $r,r'$ with $-(1/2){\leq}r'{<}r,$ $r{\geq}0,$ there is
 a constant $C_{r,r',\pq}{>}0$ such that 
 \begin{equation}\label{e5.3a}
 \|\uq\|_{r'+(1/2)}\leq C_{r,r'\pq}\|\bar{\partial}^{\Ft}_{\sfM}\uq\|^{2}_{r} ,\;\;\;
 \forall \uq\in 
 \Fq^{\pq,0}(\sfM){\cap}(\cN_{\Ft}^{\pq,0}(\sfM))^{\perp}.
\end{equation}\end{itemize} 
 \end{prop} 
\subsection{Deformation of tangential $CR$ complexes} \label{sub5.3}
Let us consider the situation described at the beginning of \S\ref{sec5.2}.
Here we make the additional assumption that $\sfM$ is a generic
compact smooth 
$CR$ submanifold of a complex variety $\sfX$ of 
dimension $n{+}k$ 
and that the complex $CR$ line bundle $\Ft_{\sfM}$ is the pullback
of a holomorphic line bundle $\Ft_{\sfX}$ on $\sfX,$ associated to a global
meromorphic function on $\sfX,$ whose pole divisor has support $D.$
We indicate by $\sfM_{0}$ the $CR$ structure, of type $(n,k),$ 
induced on $\sfM$ by the embedding $\sfM\,{\hookrightarrow}\,\sfX$ 
and consider a smooth one-parameter deformation $\{\sfM_{\ttt}\}$ 
on $\sfM$ of
this $CR$ structure $\sfM_{0},$ with the constraint:  
\begin{equation}
 \text{The $CR$ structures of $\sfM_{\ttt}$ agree to infinite order on $\sfM\cap{D}.$}
\end{equation}
 \par 
 We observed in \S\ref{sub3.2} that in this case $\Ft_{\sfM}$ can be also considered
 as a $CR$ line bundle on each $\sfM_{\ttt}.$ We need to  
 consider the different tangential
 Cauchy-Riemann complexes with coefficients in $\Ft_{sfM}$ as 
 complexes of partial differential operators 
 acting on the same spaces of vector valued functions. 
 \par 
 We showed in \S\ref{sub3.1a} that  we can represent the tangential $CR$ complex
 (on the trivial line bundle) by fixing a complement $\Et_{\sfM}{=}\Et^{0,1}_{\sfM}$ of 
 $\Ta^{1,0}_{\sfM}.$ This yields isomorphisms between 
 $\Et_{\sfM}^{\pq,\qq}{=}\Lambda^{\qq}(\Et_{\sfM}){\wedge_{\sfM}}\Lambda^{\pq}(\Ta^{1,0}_{\sfM})$
and $\Qt^{\pq,\qq}_{\sfM}$ for all integers $\pq,\qq$ with 
$0{\leq}\pq{\leq}n{+}k,$ $0{\leq}\qq{\leq}n.$ \par
We will be interested in deformations $\sfM_{\ttt}$ corresponding to \textit{small} values
of the parameter $\ttt.$ In particular, after shrinking, we can assume that 
$\Et_{\sfM}$ is transversal to $\Ta^{1,0}_{\sfM_{\ttt}}$ for all values of the parameter $\ttt.$ 
Then the tangential $CR$ complexes
\begin{equation*} \vspace{3pt}
\begin{CD}
 0 @>>>\Eq^{0,0}_{\sfM}@>{\sfd''_{\sfM_{\ttt}}}>> \Eq^{0,1}_{\sfM}@>{\sfd''_{\sfM_{\ttt}}}>> \cdots\to
 \Eq^{0,n-1}_{\sfM}
 @>{\sfd''_{\sfM_{\ttt}}}>> \Eq^{0,n}_{\sfM}@>>>0
\end{CD}
\end{equation*}
are defined as complexes of linear partial
 differential operators, smoothly depending on the parameter
$\ttt,$ which act on the germs of sections of the same fiber bundles. \par
To obtain an analogous presentation for positive values of $\pq,$ we note that 
every $\tauup$ in $\Ta^{1,0}_{\sfM_{\ttt}}$ uniquely decomposes into the sum
$\tauup'{+}\tauup''$ of a $\tauup'\,{\in}\,\Ta^{1,0}_{\sfM_{0}}$ and $\tauup''\,{\in}\,\Et_{\sfM}.$ 
The correspondence $\tauup\,{\leftrightarrow}\,\tauup'$ 
yields smooth isomorphisms $\Qt^{\pq,0}_{\sfM_{\ttt}}\,{\simeq}\,\Qt^{\pq,0}_{\sfM_{0}}$
of complex vector bundles, from which we obtain 
also for positive $\pq$ 
representations 
\begin{equation*} \vspace{3pt}
\begin{CD}
 0 @>>>\Eq^{\pq,0}_{\sfM}@>{\sfd''_{\sfM_{\ttt}}}>> \Eq^{\pq,1}_{\sfM}@>{\sfd''_{\sfM_{\ttt}}}>> \cdots\to
 \Eq^{0,n-1}_{\sfM}
 @>{\sfd''_{\sfM_{\ttt}}}>> \Eq^{\pq,n}_{\sfM}@>>>0
\end{CD}
\end{equation*}
of the tangential $CR$ complexes on $\sfM_{\ttt}$  by
linear partial differential operators  acting on the same complex vector bundles and smoothly
depending on the parameter $\ttt.$ \par
Finally, for the line bundle $\Ft_{\sfM},$ we use the trivialization given by the sections $\sigmaup_{i}$
of \S\ref{sub3.2} and \eqref{e3.3} to compute their $\bar{\partial}^{\,\Ft}_{\sfM_{\ttt}}.$ 
The form $\phiup_{i}^{(\ttt)}\,{\in}\,\Qq^{0,1}(U_{i})$ decomposes into the sum
of a $\psiup_{i}^{(\ttt)}\,{\in}\,\Eq^{0,1}(U_{i})$ and a form in $\Gamma^{\infty}(U_{i},\Ta^{1,0}_{\sfM_{\ttt}}).$ 
\par 
Then we can define for every $0{\leq}\pq{\leq}n{+}k$ and $0{\leq}\qq{<}n$ and 
all values of the parameter $\ttt$ a  
linear partial differential operators $\sfd''_{\sfM_{\ttt}}{}^{\!\!\!\!\Ft}$ 
mapping germs of smooth sections of 
$\Ft_{\sfM}\,{\otimes}\,\Et^{\pq,\qq}_{\sfM}$ to 
germs of smooth sections of 
$\Ft_{\sfM}\,{\otimes}\,\Et^{\pq,\qq+1}_{\sfM}$ in such a way that 
\begin{equation*}
\sfd''_{\sfM_{\ttt}}{}^{\!\!\!\!\Ft}(\alphaup\cdot\sigmaup_{i})=(\sfd''_{\sfM_{\ttt}}\alphaup+
(-1)^{\pq+\qq}\psiup_{i}^{(\ttt)}\wedge\alphaup)\cdot\sigmaup_{i},\;\; \forall i,
\forall \alphaup\in\Eq^{\pq,\qq}_{\sfM}(U_{i}).
\end{equation*}
In this way we represent the different tangential $CR$ complexes for forms with coefficients in
$\Ft_{\sfM}$ as complexes of differential operators, smoothly depending on the parameter
$\ttt,$ but acting on the same smooth complex vector bundles: 
\begin{equation} \label{e5.6}
\vspace{8pt}
\begin{aligned}
\begin{CD}
 0 \to \Eq^{\pq,0}(\sfM,\Ft_{\sfM}) @>{\sfd''_{\sfM_{\ttt}}{}^{\!\!\!\!\Ft}}>> \Eq^{\pq,1}(\sfM,\Ft_{\sfM})
 \to \cdots\end{CD}\qquad\qquad\qquad\\
 \begin{CD}
   \cdots  \to 
 \Eq^{\pq,n-1}(\sfM,\Ft_{\sfM}) @>{\sfd''_{\sfM_{\ttt}}{}^{\!\!\!\!\Ft}}>> \Eq^{\pq,n}(\sfM,\Ft_{\sfM})
 \to 0. 
\end{CD}
\end{aligned}
\end{equation}

\subsection{Estimates for deformations}
We keep the notation and assumptions of \S\ref{sub5.3}. Our goal is to show 
that the results of Propositions~\ref{p5.2},\ref{p5.3}
remain valid for small deformations. 
\begin{prop}\label{p5.4}
 Let $\sfM$ be a compact  generic $CR$ submanifold 
of type $(n,k)$ of a complex variety $\sfX.$ 
Let $\Ft_{\sfX}$ be the
holomorphic line bundle associated to a global meromorphic function on $\sfX$
and $D$ the support of its polar divisor. We assume that $\sfM$ is $2$-pseudoconcave.
Let $\{\sfM_{\ttt}\}$ be a deformation of the $CR$ structure induced on $M{=}M_{0}$
by the embedding into $\sfX,$ smoothly depending on the parameter $\ttt$
and with $CR$ structures agreeing to infinite order on~$D{\cap}\sfM.$
Then we can find $\epi_{0}{>}0$ such that for all $|\ttt|{<}\epi_{0}$ 
\begin{enumerate}
 \item for every 
  $0{\leq}\pq{\leq}n{+}k$ and every  
 real $r{\geq}0$ there
 is a positive constant $c_{r,\pq}{>}0$ such that 
\begin{equation} \label{e5.2a}\begin{cases}
 c_{r,\pq}\cdot\|\uq\|^{2}_{r+(1/2)}\leq \|\sfd''_{\sfM_{\ttt}}{}^{\!\!\!\!\Ft}\uq\|^{2}_{r}
+\|\uq\|_{0}^{2},
 \;\; \forall \uq\in 
 \Fq^{\pq,0}(\sfM)\\[4pt]
 c_{r,\pq}\cdot\|\vq\|^{2}_{r+(1/2)}\leq \|\sfd''_{\sfM_{\ttt}}{}^{\!\!\!\!\Ft}\vq\|^{2}_{r}
 + \|(\sfd''_{\sfM_{\ttt}}{}^{\!\!\!\!\Ft})^{\ast}\vq\|^{2}_{r}+\|\vq\|_{0}^{2},
 \;\; \forall \vq\in 
 \Fq^{\pq,1}(\sfM).
 \end{cases}
\end{equation}
\item for every 
  pair $r',r$ of real numbers with $-(1/2){\leq}r'{<}r,$ 
$r{\geq}0$ and  \par\noindent
$0{\leq}\pq{\leq}n{+}k,$ 
there
 is a positive constant $C_{r,r,'\pq}{>}0$ such that 
\begin{equation} \label{e5.2b} \begin{cases}
 \|\uq\|^{2}_{r'+(1/2)}\leq C_{r,r',\pq}
 \|\sfd''_{\sfM_{\ttt}}{}^{\!\!\!\!\Ft}\uq\|^{2}_{r}
\;\;\;
 \forall \uq\in 
 \Fq^{\pq,0}(\sfM)\cap (\cN^{\pq,0}_{\Ft}(\sfM_{\ttt}))^{\perp}\\[4pt]
\begin{aligned}
 \|\vq\|^{2}_{r'+(1/2)}\leq C_{r,r',\pq}\left(
 \|\sfd''_{\sfM_{\ttt}}{}^{\!\!\!\!\Ft}\vq\|^{2}_{r}
 + \|(\sfd''_{\sfM_{\ttt}}{}^{\!\!\!\!\Ft})^{\ast}\vq\|^{2}_{r}\right)
 \;\; \qquad\qquad\qquad\\
 \forall \vq\in 
 \Fq^{\pq,1}(\sfM)\cap (\cN^{\pq,1}_{\Ft}(\sfM_{\ttt}))^{\perp}.
 \end{aligned}
 \end{cases}
\end{equation}
\item 
$\dim_{\C}(H^{\pq,0}(\sfM_{\ttt},\Ft_{\sfM})){\leq}
\dim_{\C}(H^{\pq,0}(\sfM_{0},\Ft_{\sfM})){<}+\infty$
and \par \noindent
$\dim_{\C}(H^{\pq,1}(\sfM_{\ttt},\Ft_{\sfM})){\leq}
\dim_{\C}(H^{\pq,1}(\sfM_{0},\Ft_{\sfM})){<}+\infty$
for all $0{\leq}\pq{\leq}n{+}k.$ 
\end{enumerate}
\end{prop} 
\begin{proof}
The proof of the subelliptic estimates in \cite{HN1} yield \eqref{e5.2a}, the constants being 
uniform for $|\ttt|{<}\epi_{0}$ because they depend upon the coefficients of the linear differential
operators $\sfd''_{\sfM_{\ttt}}{}^{\!\!\!\!\Ft}$ and their derivatives. \par \smallskip 
We prove (2) and (3) in the case of forms ($\qq\,{=}\,1$), as the case of sections ($\qq{=}0$)
is analogous and even simpler. \par 
We begin by proving that, for any pair $r,r'$ with $(-1/2){<}r'{<}r,$  $r{\geq}0$ and for
all $0{\leq}\pq{\leq}
 n{+}k,$ we can find $\epi>0$ and $C_{r,r',\pq}{>}0$ such that 
 \begin{equation}\label{e5.9}
\begin{aligned}
 \|\vq\|_{r'{+}(1/2)}^{2}\leq C_{r,r',\pq}\left(\|\sfd''_{\sfM_{\ttt}}{}^{\!\!\!\!\Ft} \vq\|^{2}_{r}+
 \|(\sfd''_{\sfM_{\ttt}}{}^{\!\!\!\!\Ft})^{*} \vq\|^{2}_{r}\right),\;\;\;\;\qquad\qquad
 \\
 \forall |\ttt|<\epi,\;\;\forall\vq\in\Eq^{\pq,1}(\sfM,\Ft_{\sfM})\cap(\cN^{\pq,1}_{\Ft}(\sfM_{0}))^{\perp}.
\end{aligned}
\end{equation}
We argue by contradiction: if this was false, we could find a sequence $\{\vq_{\nuup}\}$
in $\Eq^{\pq,1}(\sfM,\Ft_{\sfM})\cap(\cN^{\pq,1}_{\Ft}(\sfM_{0}))^{\perp}$ and a sequence 
$\{\ttt_{\nuup}{\to}0\}$ such that 
\begin{equation*}
 \|\vq_{\nuup}\|_{r'{+}(1/2)}^{2}=1,\;\;\; \|\sfd''_{\sfM_{\ttt_{\nuup}}}{}^{\!\!\!\!\Ft} \vq_{\nuup}\|^{2}_{r}+
 \|(\sfd''_{\sfM_{\ttt_{\nuup}}}{}^{\!\!\!\!\Ft})^{*} \vq_{\nuup}\|^{2}_{r}<2^{-\nuup}.
\end{equation*}
By (1), the sequence $\{\vq_{\nuup}\}$ is uniformly bounded in $\sfW^{r+{1/2}}$ and then,
by the compactness of the inclusion $\sfW^{r+(1/2)}\,{\hookrightarrow}\,\sfW^{r'+(1/2)},$ 
contains a subsequence which strongly converges in $\sfW^{r'+(1/2)}$ to a $\vq_{\infty},$
which on one hand should be
different from $0$ because $\|\vq_{\infty}\|_{r'+(1/2)}{=}1,$ 
on the other hand should be zero because it belongs to $\cN_{\Ft}^{\pq,1}(\sfM_{0})\,{\cap}\,
(\cN_{\Ft}^{\pq,1}(\sfM_{0}))^{\perp}.$ \par 
\smallskip
Estimate \eqref{e5.9} implies that $\cN_{\Ft}^{1,\pq}(\sfM_{\ttt})\,{\cap}\,
(\cN^{\pq,1}_{\Ft}(\sfM_{0}))^{\perp}\,{=}\,\{0\}$ for $|\ttt|{<}\epi.$ 
This yields (3), because the dimension of $H^{1,\pq}(\sfM_{\ttt},\Ft)$
equals that of $\cN_{\Ft}^{1,\pq}(\sfM_{\ttt}).$   \par\smallskip
We can prove (2) by an argument similar to the one used for \eqref{e5.9}.
\end{proof}
\subsection{Vanishing theorems}
We will use 
Proposition \ref{p5.4} to solve  
${\sfd''}^{\Ft}_{\!\!\!\sfM_{\ttt}}\uq_{\ttt}{=}f,$ for forms $f$ in $\Eq^{0,1}(\sfM,\Ft_{\sfM}),$
with $\uq_{\ttt}$ satisfying uniform estimates in the So\-bolev norms
$\sfW^{r},$ for sufficiently large $r,$ yielding by the 
immersion theorems control upon the derivatives of $\uq_{\ttt}.$ 
\begin{thm}\label{t5.5}
 Let us keep the assumptions of Proposition~\ref{p5.4}. Assume moreover that $\Ft_{\sfM}$ is 
 a positive line bundle on $\sfX.$ Then we can find $\epi_{0}>0$ such that 
\begin{equation}
 H^{\pq,\qq}(\sfM_{\ttt},\Ft_{\sfM}^{-\ell})=0,
  \;\;\forall \ell\geq\ell_{0},\; 0\leq\pq\leq{n}+k,\;\qq=0,1,\;\; |\ttt|<\epi_{0}.
\end{equation}
Moreover, for every $\fq\,{\in}\,\Eq^{0,1}(\sfM,\Ft_{\sfM}^{-\ell})$ satisfying ${\sfd''}^{\Ft}_{\!\!\!\sfM_{\ttt}}f\,{=}\,0$
there is a unique $\uq_{\ttt}$ in $\Gamma^{\infty}(\sfM,\Ft^{-\ell})$ such that, for $r'{<}r$ and $r{\geq}1/2,$   
\begin{equation}\label{e5.11}
 {\sfd''}^{\Ft}_{\!\!\!\sfM_{\ttt}}\uq_{\ttt}{=}\fq\;\;\;\text{and}\;\;\; \|\uq_{\ttt}\|_{r'}\leq C_{r',r}\|\fq\|_{r-(1/2)}.
\end{equation}
\end{thm} 
\begin{proof} The first statement is a consequence of Theorem~\ref{vanish} and
Proposition~\ref{p5.4}. These results allow us to fix $\epi_{0}$ in such a way that,
by \eqref{e5.2b}, 
\begin{equation*}
 \|| \vq\||_{\ttt} = 
  \left(\| {\sfd''}^{\Ft}_{\!\!\!\sfM_{\ttt}}\vq\|^{2}_{0}+\|( {\sfd''}^{\Ft}_{\!\!\!\sfM_{\ttt}})^{*}\vq\|_{0}\right)^{1/2}
\end{equation*}
is a pre-Hilbertian norm on $\Et^{0,1}(\sfM_{\ttt},\Ft_{\sfM}^{-\ell}).$ 
Let  $\mathbf{S}_{\ttt}$ be its completion  with respect to this norm. Given 
$f\,{\in}\,\Et^{0,1}(\sfM,\Ft_{\sfM}^{-\ell}),$ we can consider $\vq{\to}(\vq|f)_{0}$  as
a linear continuous functional on $\mathbf{S}_{\ttt}$ and hence  by Riesz representation 
Theorem
there is a unique 
$\wq\,{\in}\,\mathbf{S}_{\ttt}$ such that 
\begin{equation*}
 (\vq|f)_{0}=( {\sfd''}^{\Ft}_{\!\!\!\sfM_{\ttt}}\vq| {\sfd''}^{\Ft}_{\!\!\!\sfM_{\ttt}}\wq)_{0}
 +(( {\sfd''}^{\Ft}_{\!\!\!\sfM_{\ttt}})^{*}\vq | ( {\sfd''}^{\Ft}_{\!\!\!\sfM_{\ttt}})^{*}\wq)_{0},\;\;
 \forall \vq\in \Et^{0,1}(\sfM,\Ft_{\sfM}^{-\ell}).
\end{equation*}
By the subelliptic estimate \eqref{e5.2a}, in fact 
$\wq\,{\in}\,\Et^{0,1}(\sfM,\Ft^{-\ell}_{\sfM}),$
i.e. $\wq$ is smooth, because we assumed that 
$f$ is smooth. If $f$ satisfies the integrability condition 
${\sfd''}^{\Ft}_{\!\!\!\sfM_{\ttt}}f=0,$
then integration by parts yields
\begin{align*}
 0&=({\sfd''}^{\Ft}_{\!\!\!\sfM_{\ttt}}\wq|{\sfd''}^{\Ft}_{\!\!\!\sfM_{\ttt}}f)_{0}=
 (({\sfd''}^{\Ft}_{\!\!\!\sfM_{\ttt}})^{*}{\sfd''}^{\Ft}_{\!\!\!\sfM_{\ttt}}\wq|f)\\
 &=
 ( {\sfd''}^{\Ft}_{\!\!\!\sfM_{\ttt}}({\sfd''}^{\Ft}_{\!\!\!\sfM_{\ttt}})^{*}{\sfd''}^{\Ft}_{\!\!\!\sfM_{\ttt}}\wq| 
 {\sfd''}^{\Ft}_{\!\!\!\sfM_{\ttt}}\wq)_{0}
 +(( {\sfd''}^{\Ft}_{\!\!\!\sfM_{\ttt}})^{*}({\sfd''}^{\Ft}_{\!\!\!\sfM_{\ttt}})^{*}{\sfd''}^{\Ft}_{\!\!\!\sfM_{\ttt}}\wq | 
 ( {\sfd''}^{\Ft}_{\!\!\!\sfM_{\ttt}})^{*}\wq)_{0}\\
 &= (({\sfd''}^{\Ft}_{\!\!\!\sfM_{\ttt}})^{*}{\sfd''}^{\Ft}_{\!\!\!\sfM_{\ttt}}\wq|
 ({\sfd''}^{\Ft}_{\!\!\!\sfM_{\ttt}})^{*}{\sfd''}^{\Ft}_{\!\!\!\sfM_{\ttt}}\wq)_{0} \;\; \Longrightarrow\;\; 
 ({\sfd''}^{\Ft}_{\!\!\!\sfM_{\ttt}})^{*}{\sfd''}^{\Ft}_{\!\!\!\sfM_{\ttt}}\wq=0
 \\
 &\Longrightarrow 0=(({\sfd''}^{\Ft}_{\!\!\!\sfM_{\ttt}})^{*}{\sfd''}^{\Ft}_{\!\!\!\sfM_{\ttt}}\wq|\wq)_{0}=
 \|{\sfd''}^{\Ft}_{\!\!\!\sfM_{\ttt}}\wq\|_{0}^{2}\;\;\;\Longrightarrow {\sfd''}^{\Ft}_{\!\!\!\sfM_{\ttt}}\wq=0.
\end{align*}
 By \eqref{e5.2b}, for all
$0 < r < 1/2$
\begin{align*}
 (\uq_{\ttt}|\uq_{\ttt})_{0} & = (({\sfd''}^{\Ft}_{\!\!\!\sfM_{\ttt}})^{*}\wq,({\sfd''}^{\Ft}_{\!\!\!\sfM_{\ttt}})^{*}\wq)_{0} = (\wq, {\sfd''}^{\Ft}_{\!\!\!\sfM_{\ttt}}({\sfd''}^{\Ft}_{\!\!\!\sfM_{\ttt}})^{*}\wq)
=  (\wq,f)_{0}\\
 &\leq  \|f\|_{-r}\|\wq\|_{r} \leq {C}''_{r}\|f\|_{-r}\|\uq_{\ttt}\|_{r-\frac{1}{2}} \leq {C}''_{r}\|f\|_{-r}\|\uq_{\ttt}\|_{0} 
\end{align*}
Hence 
\begin{equation*}
  \uq_{\ttt}= ({\sfd''}^{\Ft}_{\!\!\!\sfM_{\ttt}})^{*}\wq\;\;\;\text{solves}\;\;\;
  {\sfd''}^{\Ft}_{\!\!\!\sfM_{\ttt}}\uq_{\ttt}=f\;\;\;\text{and}\;\;\; \|\uq_{\ttt}\|_{0}\leq{C}''_{r}\|f\|_{-r},\;\;
  \forall 0 < r < 1/2,
\end{equation*}
with constants $C_{r}''$ independent of $|\ttt|{<}\epi_{0}.$
Since $\cO(\sfM_{\ttt},\Ft_{\sfM})=\{0\},$ the solution is unique and therefore we obtain the
estimates in \eqref{e5.11} by recurrence, using, for any pseudodifferential operator $P$
on $\sfM,$  the identity
\begin{equation*}
 {\sfd''}^{\Ft}_{\!\!\!\sfM_{\ttt}}P(u_{\ttt})=[{\sfd''}^{\Ft}_{\!\!\!\sfM_{\ttt}},P](u_{\ttt})+P(f).
\end{equation*}
 If $P$ has order $r',$ with $0{<}r'{<}1/2,$ then 
\begin{equation*}
 \| P(u_{\ttt})\|_{0}\leq C''_{r'}\left(\|[{\sfd''}^{\Ft}_{\!\!\!\sfM_{\ttt}},P](u_{\ttt})\|_{-r'}+\|P(f)\|_{-r'}\right)
 \leq \text{const}_{P} C''_{r'} \left(\|u_{\ttt}\|_{0}+\|f\|_{0}\right),
\end{equation*}
with a constant $\text{const}_{P}$ which only depends on $P.$ This yields 
\begin{equation*}
 \|\uq_{\ttt}\|_{r'} \leq C_{r',(1/2)}\|f\|_{0}.
\end{equation*}
with constants independent of $\ttt.$ 
Repeating this argument we obtain the estimate in the statement.
\end{proof}

\section{Proof of Theorem \ref{main}}\label{sec6}
\subsection{The general case}\label{sub6.1}
We keep the notation of \S\ref{sub3.2}. The $(f''_{i})$ define a global section $\sigmaup$
of $\Ft_{\sfM},$ which is $CR$ for the structure on $\sfM_{0}$ and is $\neq{0}$ on $\sfM{\backslash}D.$
\par
Fix an integer $\ell$ such that $\Ft_{\sfX}^{\ell}$ is very ample and
$H^{0,1}(\sfM_{\ttt},\Ft^{-h}_{\sfM}){=}0$ for $h{\geq}\ell$ and $|\ttt|{<}\epi_{0}$, which is possible by Theorem \ref{t5.5}.

For every $\ttt$ we
set 
\begin{equation*}
 \gq_{\ttt}= 
\begin{cases}
  \sfd''_{\sfM_{\ttt}}{}^{\!\!\!\Ft^{-\ell}}\sigmaup^{-\ell}, &\text{on $\sfM{\backslash}D,$}\\
 0, &\text{on $D.$}
\end{cases}
\end{equation*}
Since the $CR$ structure 
of $\sfM_{\ttt}$ agrees with that of $\sfM_{0}$ to infinite order on
$D,$ the $\gq_{\ttt}$'s are smooth sections in $\Eq^{0,1}(\sfM,\Ft^{-\ell})$ that vanish to infinite
order on $D$ and satisfy the integrability condition $ \sfd''_{\sfM_{\ttt}}{}^{\!\!\!\Ft^{-\ell}}\gq_{\ttt}\,{=}\,0.$
By Theorem~\ref{t5.5}, for $|\ttt|{<}\epi_{0},$ 
 we can find $\vq_{\ttt}\,{\in}\,\Gamma^{\infty}(\sfM,\Ft_{\sfM}^{-\ell})$ to solve 
\begin{equation*}
  \sfd''_{\sfM_{\ttt}}{}^{\!\!\!\Ft^{-\ell}}\vq_{i,\ttt}=\gq_{\ttt},\;\;
  \text{and we have}\;\; \|\vq_{\tauup}\|_{r'}\leq C_{r',r}\|\gq_{\ttt}\|_{r-(1/2)}
\end{equation*}
for $0{\leq}r'{<}r,$ with a constant $C_{r,r'}$ independent of $\ttt.$ 
Then 
\begin{equation*}
\sigmaup^{-\ell}-\vq_{\ttt}\in\cO_{\sfM_{\ttt}}(\sfM_{\ttt}\backslash D,\Ft^{-\ell}_{\sfM}).
\end{equation*}
By Sobolev's embedding theorems, the $\sup$-norms of the
first derivatives of the
$\vq_{i,\ttt}$ are bounded by their $\sfW^{r'}$-norms for $r'{>}(2n{+}k{+}2)/2.$ 
As $\|\gq_{i,\ttt}\|_{r-(1/2)}{\to}0$ for $\ttt{\to}0,$ we obtain that, 
for
$|\ttt|{\ll}{1},$ the sections 
\begin{equation*}
\etaup_{\ttt} = (\sigmaup^{-\ell}-\vq_{\ttt})^{-1}\in\cO_{\sfM_{\ttt}}(\sfM_{\ttt}\backslash D,\Ft^{\ell}_{\sfM}).
\end{equation*}
extend to smooth  sections of $\cO_{\sfM_{\ttt}}(\sfM_{\ttt},\Ft^{\ell}_{\sfM})$, defining the divisor $D$.

Since $\Ft_{\sfX}^{\ell}$ is very ample, there are
sections $\sigmaup_{0},\sigmaup_{1},\hdots,\sigmaup_{m}{\in}\cO_{\sfM_{0}}(\sfX,\Ft^{\ell}_{\sfX})$
providing 
a holomorphic embedding of $\sfX$ into $\CP^{m}.$ \par
Their restrictions $(\sigmaup_{0.0},\sigmaup_{1,0},\hdots,\sigmaup_{m,0})$
to $\sfM$ provide 
a $CR$ embedding of $\sfM_{0}$ into $\CP^{m}.$ For every $\ttt$ and $1{\leq}i{\leq}m$
set 
\begin{equation*}
 \fq_{i,\ttt}= 
\begin{cases}
  \sfd''_{\sfM_{\ttt}}{}^{\!\!\!\Ft^{-\ell}}(\etaup_{\ttt}^{-2}\sigmaup_{i}), &\text{on $\sfM{\backslash}D,$}\\
 0, &\text{on $D.$}
\end{cases}
\end{equation*}
Then, arguing as above and invoking Theorem~\ref{t5.5}
 for $|\ttt|{<}\epi_{0},$ 
 we can find $\uq_{i,\ttt}\,{\in}\,\Gamma^{\infty}(\sfM,\Ft_{\sfM}^{-\ell})$ to solve 
\begin{equation*}
  \sfd''_{\sfM_{\ttt}}{}^{\!\!\!\Ft^{-\ell}}\uq_{i,\ttt}=f_{i,\ttt},\;\;
  \text{and we have}\;\; \|\uq_{i,\tauup}\|_{r'}\leq C_{r',r}\|f_{i,\ttt}\|_{r-(1/2)}
\end{equation*}
for $0{\leq}r'{<}r,$ with a constant $C_{r,r'}$ independent of $\ttt.$ 
Then 
\begin{equation*}
 \sigmaup_{i,\ttt}= \sigmaup_{i}-\uq_{i,\ttt}\etaup_{\ttt}^{2}\in\cO_{\sfM_{\ttt}}(\sfM_{\ttt},\Ft^{\ell}_{\sfM}).
\end{equation*}
As $\|f_{i,\ttt}\|_{r-(1/2)}{\to}0$ for $\ttt{\to}0,$ we obtain that, 
for
$|\ttt|{\ll}{1},$ the sections \\
$(\sigmaup_{0,\ttt},\sigmaup_{1,\ttt},\hdots,\sigmaup_{m,\ttt}),$
being a small $\Cc^{1}$-perturbation of
$(\sigmaup_{0,0},\sigmaup_{1,0},\hdots,\sigmaup_{m,0}),$
still provide a $CR$ immersion of $\sfM_{\ttt}$ into $\CP^{m}.$  \par
The last part of the statement follows from \cite{HNmero}, where it is shown that
the maximum degree of transcendence of the field of $CR$ meromorphic functions on $\sfM_{\ttt}$
is $n{+}k$: this implies that  all pseudoconcave 
$CR$ manifold of type $(n,k)$ having a projective embedding, can
be embedded into a projective complex variety of 
dimension~$n{+}k.$ 

\subsection{Generic $CR$ submanifolds of the projective space}
The line bundles on the projective space $\CP^{n+k}$ are parametrized 
(modulo equivalence) by the integers. They are all holomorphically equivalent
to integral powers 
of the line bundle   
$\Ot_{\CP^{n+k}},$ which can be described,
in the covering $\{U_{i}\,{=}\,\{z_{i}{\neq}0\}\}$ 
by the
transition functions  $\{g_{i,j}\,{=}z_{j}^{-1}z_{i}\}.$
The support of the pole divisor of a meromorphic function $\textsf{f}$ on $\CP^{n+k}$
is the set of common zeros of a homogeneous polynomial 
\begin{equation}\label{e6.1}
 D=\{\wp(z_{0},z_{1},\hdots,z_{n+k})=0\}.
\end{equation}
If $\wp(z)$ has degree $d,$ this is the zero locus of a section of 
$\Ot_{\CP^{n+k}}^{\,d}.$ 
 Replacing
 $\wp$ by one of its powers if necessary,
we can as well assume that~$d{>}n{+}k.$ 
Then we have
\begin{equation*}
 H^{\qq}(\CP^{n+k},\Ot_{\CP^{n+k}}^{d})=\begin{cases}
\{ \psiup\in\C[z_{0},z_{1},\hdots,z_{n+k}]\mid \psiup(\lambdaup{z})=\lambdaup^{d}\psiup(z)\}, 
\;\;
\text{if $\qq{=}0,$}\\
 \{0\}\;\;
 \text{for $\qq{>}0.$}\end{cases}
\end{equation*}
\par \smallskip
The fact that the pullback 
$\Ot^{\,d}_{\sfM}$ of $\Ot^{\,d}_{\CP^{n+k}}$ is a $CR$ line bundle on $\sfM_{\ttt}$ 
implies that
also all pullbacks $\Ot^{\,\pq}_{\sfM}$ of $\Ot^{\,\pq}_{\CP^{n+k}}$ to $M,$ for any $\pq\,{\in}\,Z,$
are $CR$ line bundles on $\sfM_{\ttt}.$  Indeed the fact
that $\Ot^{\,d}_{\sfM}$ is a $CR$ line bundle
on $\sfM_{\ttt}$ means that we can find an open covering of $\sfM,$ that we can take of
the form $\{U_{i,\muup}\}$ 
with $U_{i,\muup}\,{\subseteq}\,U_{i}\,{=}\,\{z_{i}{\neq}0\}$ and $\sfd''_{\sfM_{\ttt}}$-closed
forms $\phiup_{i,\muup}\,{\in}\,\Eq^{0,1}(U_{i,\muup})$ such that 
\begin{equation*}
 \left(\frac{z_{j}}{z_{i}}\right)^{d}\sfd''_{\sfM_{\ttt}} \left(\frac{z_{i}}{z_{j}}\right)^{d}
 = \phiup_{i,\muup}-\phiup_{j,\nuup}\;\;\;\text{on $U_{i,\muup}\cap{U}_{j,\nuup}.$}
\end{equation*}
This is equivalent to 
\begin{equation*}
 \left(\frac{z_{j}}{z_{i}}\right)\sfd''_{\sfM_{\ttt}} \left(\frac{z_{i}}{z_{j}}\right)
 = \frac{1}{d}\phiup_{i,\muup}-\frac{1}{d}\phiup_{j,\nuup}\;\;\;\text{on $U_{i,\muup}\cap{U}_{j,\nuup},$} 
\end{equation*}
showing that $\Ot_{\sfM},$ and therefore all $\Ot_{\sfM}^{\pq}$ with $\pq\,{\in}\,\Z$ 
are $CR$ line bundles on~$\sfM_{\ttt}.$ 
\par \smallskip
By repeating the proof in \S\ref{sub6.1} by using the sections $\sigmaup_{i}\,{=}\,z_{i}$ of
$\Ot_{\sfM},$ we obtain 
\begin{thm}  \label{projective}
 Let $\sfM$ be a smooth  
 compact generic $CR$ submanifold of type $(n,k)$ of $\CP^{n+k}$ and $D$ the zero locus
 in $\CP^{n+k}$ of a homogeneous polynomial of degree $d$ in $\C[z_{0},z_{1},\hdots,z_{n+k}].$
 Let $\{\sfM_{\ttt}\}$ be a family of $CR$ structures of type $(n,k)$ on $\sfM,$ smoothly 
 depending on a real parameter $\ttt.$ If 
\begin{enumerate}
 \item $\sfM_{0}$ is induced by the embedding $\sfM\,{\hookrightarrow}\,\CP^{n+k}$;
 \item $\sfM_{0}$ is $2$-pseudoconcave;
 \item the $CR$ structures of all $\sfM_{\ttt}$ agree to infinite order on $\sfM\,{\cap}\,D.$
\end{enumerate}
Then we can find $\epi_{0}>0$ such that, for every $|\ttt|<\epi_{0},$ $\sfM_{\ttt}$ admits a
generic $CR$ embedding into $\CP^{n+k}.$ \qed
\end{thm}
\subsection{Generic $CR$ submanifolds of Fano varieties}
We recall that a complex variety $\sfX$ is \emph{Fano} if its anticanonical bundle
$\Kt_{\sfX}^{-1}$ is ample.  If we add to the assumptions of Theorem~\ref{main}
the requirement that $\sfX$ is Fano, then we can define a projective embedding
of $\sfM_{0}$ by using $CR$ sections $\sigmaup_{0},\sigmaup_{1},\hdots,\sigmaup_{m}$
of a power $\Kt_{\sfM}^{-\ell}$ of the pullback on $\sfM$ of the anticanonical bundle. Indeed,
by \cite{Kl}, the tensor product $\Kt_{\sfX}^{-\ell}\,{\otimes}\,\Ft_{\sfX}^{-k}$ is the dual
of a positive bundle and we can apply the arguments in \S\ref{sub6.1} to obtain 
sections $\sigmaup_{i,\ttt}$ in
$\cO_{\sfM_{\ttt}}(\sfM_{\ttt},\Kt_{\sfM}^{-\ell}),$ which, for $|\ttt|{\ll}{1}$
yield projective $CR$ 
embeddings $\sigmaup_{\ttt}=(\sigmaup_{0,\ttt},\sigmaup_{1,\ttt},\hdots,\sigmaup_{m,\ttt})$
of $\sfM_{\ttt}$ into $\CP^{m}.$ The images $\sfM'_{\ttt}\,{=}\,\sigmaup_{\ttt}(\sfM_{\ttt})$
are generic submanifolds of complex $(n{+}k)$-dimensional 
subvarieties $\sfX_{\ttt}$ of $\CP^{m}$ which have, by construction, an ample anticanonical
bundle. Indeed we know from \cite[Thm.5.2]{HNmero} that, since the $\sfM_{\ttt}$ are
pseudoconcave, 
global $CR$ meromorphic functions on
$\sfM_{\ttt}$ are restrictions of global meromorphic functions on $\sfX_{\ttt}.$ 
Thus we have 
\begin{thm}  \label{Fano}
Add to the assumptions of Theorem~\ref{main} the fact that $\sfX$ is Fano.
Then there is $\epi_{0}{>}0$ such that for all $|\ttt|\,{<}\,\epi_{0}$ the abstract
$CR$ manifold 
$\sfM_{\ttt}$ has a generic $CR$ embedding
into a complex Fano variety $\sfX_{\ttt}.$ 
\end{thm}
\par\bigskip
\vspace{0.5cm}
\noindent
{\small {\bf Acknowledgements.} 
The first author was supported by Deutsche 
For\-schungs\-ge\-mein\-schaft
(DFG, German Research Foundation, \mbox{grant BR 3363/2{-}2}).}

\bibliographystyle{amsplain}
\renewcommand{\MR}[1]{}
\providecommand{\bysame}{\leavevmode\hbox to3em{\hrulefill}\thinspace}
\providecommand{\MR}{\relax\ifhmode\unskip\space\fi MR }

\end{document}